\documentclass[11pt]{article}

\usepackage{amsmath, amssymb, amsthm}
\usepackage{mathrsfs} 
\usepackage{bm}       
\usepackage{bbm}      
\usepackage{hyperref} 
\usepackage{enumitem} 
\usepackage{geometry} 
\usepackage{xcolor}
\usepackage{tcolorbox}
\usepackage{tikz}
\usepackage{tikz-cd} 
\usepackage{mathtools}
\usepackage{amsthm}
\usepackage{orcidlink}

\usepackage{mhequ}
\usepackage{mathtools}

\newtheorem{theorem}{Theorem}[section]
\newtheorem{lemma}[theorem]{Lemma}
\newtheorem{proposition}[theorem]{Proposition}
\newtheorem{corollary}[theorem]{Corollary}

\theoremstyle{definition}
\newtheorem{definition}[theorem]{Definition}

\theoremstyle{remark}
\newtheorem{remark}[theorem]{Remark}

\newtheorem{assumption}[theorem]{Assumption}

\numberwithin{equation}{section}

\usepackage{titlesec}
\titleformat{\section}{\large\bfseries}{\thesection}{1em}{}
\titleformat{\subsection}{\normalsize\bfseries}{\thesubsection}{1em}{}

\usepackage{caption}
\usepackage{natbib} 
\hypersetup{
    pdftitle={Your Paper Title},
    pdfauthor={Your Name},
}

\newcommand{\norm}[1]{\left\| #1 \right\|}
\def\<{\langle}
\def\>{\rangle}

\def\E{\mathbb E}
\newcommand{\tr}[1]{\ensuremath\trace[#1]}

\def\${|\!|\!|}

\def\D{\mathcal D}

\newcommand{\R}{\ensuremath{\mathbb{R}}}

\newcommand{\Z}{\ensuremath{\mathbb {Z}}}

\def\L{\mathcal L}

\newcommand{\eqL}{\ensuremath{\,{\buildrel \mathcal{L} \over =}\,}}

\DeclareMathOperator{\Div}{Div}
\DeclareMathOperator{\trace}{tr}
\DeclareMathOperator{\var}{var}
\def\f{\frac}

\def\var{\textrm{-var}}

\colorlet{darkblue}{blue!90!black}
\colorlet{darkred}{red!80!black}
\colorlet{darkgreen}{green!50!black}
\let\comm\comment

\theoremstyle{plain}
\newtheorem{atheorem}{Theorem}

\theoremstyle{definition}

\usepackage{orcidlink}


\newcommand{\email}[1]{\href{mailto:#1}{\texttt{#1}}}

\begin{document}

\title{Navier-Stokes with a fractional transport noise as a limit of multi-scale dynamics}

\author{
Xue-Mei Li$^{1,2\orcidlink{0000-0003-1211-0250}}$\;
Szymon Sobczak$^{1\orcidlink{0009-0006-7106-250X}}$
}
\maketitle
\vskip -10pt \noindent
{\footnotesize$^{1}$ Mathematics Institute, EPFL, Switzerland} \\
{\footnotesize$^{2}$ Department of Mathematics, Imperial, London, U.K }\\\
{\footnotesize\email{xue-mei.li, szymon.sobczak@epfl.ch}}

\begin{abstract}
    We define a bona fide rough path solution for the Navier-Stokes equation with an additional rough transport term, and show that the SPDE on the three-dimensional torus driven by a fractional Brownian motion on $H^\sigma$ has solutions characterised as the effective limits of a slow/fast system. We further show that this rough path solution is equivalent to the widely used incremental notion of solution (the unbounded rough driver formulation), demonstrating broader applicability to other nonlinear SPDEs.
 \end{abstract}
{\it keywords.} stochastic Navier-Stokes equation, rough paths,  multi-scales, Kramers-Smoluchowski limit, fractional Brownian motion

\tableofcontents

\section{Introduction}

Over the past decades, stochastic fluid equations driven by Brownian motion or Gaussian white noise have been extensively studied. Recently, there is growing interests  in auto-correlated long-range dependent noise (LRD), of which fractional Brownian motion (fBM) provides a natural model. With Gaussianity and self-similarity, the process remains sufficiently tractable to allow for rigorous analysis and at the same time remains relevant as they arise naturally as scaling limits. A non-Gaussian noises with similar properties are Hermite Processes. \cite{Embrechts-Maejima,Pipiras-Taqqu-book,Samorodnitsky}.

In this work, we study a slow/fast stochastic system modelling the evolution of a fluid on the 
$3$-dimensional torus. The slow component is governed by the Navier-Stokes equations, with an additional advection term through which the fast process enters the dynamics. 
Specifically, we consider the divergence free solutions of the equation
\begin{equ}\label{sf}
  \partial_t u^\epsilon = \nu \Delta u^\epsilon - \Pi[(u^\epsilon \cdot \nabla) u^\epsilon] - \epsilon^{H-1}  
   \Pi[(w^\epsilon \cdot \nabla) u^\epsilon],
\end{equ} 
where $\Pi$ denotes the Leray projection onto divergence-free vector fields.  We take for example the fast process a time-rescaled fractional Ornstein-Uhlenbeck process on $H^\sigma$, solving
\begin{equ}\label{fast}
    dw^\epsilon = -\epsilon^{-1} M w^\epsilon dt + 
    \epsilon^{-H} Q^{\f 12} dB^H.
\end{equ}
Here $H^\sigma$ denotes the space of mean-zero, divergence-free vector fields with $\sigma > 2 + \frac{3}{2}$, $M,Q$ are suitable linear operators, and $B_t^H$ is a cylindrical fBM on $H^\sigma$ with Hurst parameter $H\in (\f 13, 1)$. 

We show that, in the separation of scales limit  $\epsilon\to 0$, a subsequence of  rough path solutions, in the sense of Definition~\ref{def:RPsolution}, converges to the stochastic Navier-Stokes equation, 
\begin{equation}\label{eq:limit-eq}
 \partial_t u = \nu \Delta u- \Pi[(u \cdot \nabla) u] - \Pi[(M^{-1} Q^{\frac{1}{2}} d\mathbf B^H \cdot \nabla) u],
\end{equation}
with transport-type fractional Brownian noise. In addition we study the solution theory of rough stochastic Navier-Stokes equations driven by rough paths and show it is equivalent to the approximate solution theory in the unbounded rough path driver framework

In multi-scale analysis, when the Brownian driver is replaced by fractional BM, the situation becomes both substantially richer and more difficult. There is no suitable replacement for the Markov ergodic theory and  classical martingale method used to identify martingale limit cannot be used and often replaced by adhoc techniques \cite{Komorowski-Novikov-Ryzhik14,Pipiras-Taqqu00,Boufoussi-Tudor05,Friz-Gassiat-Lyons15,Fannjiang-Komorowski-2000,
Bourguin-Gailus-Spiliopoulos-typical,xu2025largedeviationprincipleslowfast,Eichinger_2020}.
 In a series of work, \cite{Gehringer-Li-fOU, gehringer-Li-19,gehringer-li2,volterra}, the authors studies time-homogenization problem with fast long range dependent rough noise.  By updating  standard functional limit theorem where the convergence is taken in the continuous topology and in distribution,  to a `rough functional central limit theorem',  the authors are able to access effective dynamics of slow/fast systems.  It is worth mentioning that the effective noise will depend on the Hermite rank of $G$.
  Rough path techniques are also used to study a complementary problem in \cite{Hairer:22} and in \cite{Li-Sieber-mild} for slow/fast stochastic partial differential equations driven by fBMs.
  
   Outside the homogenisation framework, useful references include \cite{fou_cheridito, infdimfOU,Nourdin-Nualart-Tudor,Nourdin-Nualart-Zintout-Rola,not-mixing}. There is also substantial literature on ergodicity and averaging for fractional dynamics, which is somewhat more remote from the present setting.
 We also note a growing line of work on stochastic fluid systems perturbed by fractional Brownian noise; e.g.\cite{ABCGG,numerical-frac, franco-frac}.


A  multiscale problem closely related (\ref{sf}) was investigated in \cite{DebPapMultiscale}, which serves as a source of inspiration for the present work; a shorter alternative proof of results in \cite{DebPapMultiscale} was later given in \cite{HofDeb23}. In these works, the fast motion $w^\epsilon$ is a family of Markov processes. In \cite{DebPapMultiscale},  the result is given at the level of probability distributions in an analytically weak sense.
In \cite{HofDeb23} the authors showed that any accumulation point of the slow motion is an Unbounded Rough Driver solution for a limit equation.  We shall return to  explain this in more detail later.

In contrast, the main objective here is to capture the rough creation in the transport term, showing that the effective limit points are Rough Path solutions with long-range dependence, in the sense of Definition~\ref{def:RPsolution} (and to establish the corresponding analytical toolbox). For this, it is sufficient to take the fast motion as an infinite-dimensional fractional Ornstein–Uhlenbeck process, a simple linear model making the crucial analysis underlying the creation of the effective transport fractional noise more transparent. We remark, that an additional drift as in \cite{HofDeb23} is not considered. Its treatment is disjoint from the transport noise component and would distract from the main focus.

\medskip

Next, we remark that two-scale fluid dynamics fits naturally within the multi-scale climate-weather modelling paradigm advocated by Klaus Hasselmann and has been studied extensively in the context of stochastic model reduction\cite{Majda:01} and in related fluid systems, e.g. \cite{Flandoli-pap2021,Flandoli-pap2022}. 
Our system fits also naturally within the framework of multiscale stochastic dynamics and time homogenisation. In this setting, it is often natural to view a stochastic differential equation as the limit of a family of differential equations with random coefficients. A classical example is Brownian motion (BM), which arises as the Kramers–Smoluchowski limit of  rescaled Ornstein–Uhlenbeck processes.

Before describe the convergence in the  slow/fast dynamics, we first explain the solution theory on the limiting  Navier-Stokes equation with transport rough noise.

 
For this, it is natural to look for weak solutions, testing against smooth functions from the Sobolev scale $H^s$. If one instead attempts to treat the equation as an RDE, a fundamental difficulty arises: the solution $u$, its Gubinelli derivative, and its rough path remainder $R^u$ do not naturally take values in the same space. This mismatch is intrinsic to the structure of the equation, since the drift terms involving spatial derivatives lower regularity. In particular, if $u$ is sought in $H^r$, the Gubinelli derivative suggested by the equation maps to $H^{r-1}$, preventing the direct application of standard rough path techniques. Testing against smooth functions circumvents this issue and allows us to work with real-valued rough paths.

For this reason, rough PDEs are often treated on a case-by-case basis. A general theory of rough PDEs was developed in \cite{Gubinelli-Tindel-10}, where a mild sewing lemma was introduced. In a similar spirit, a solution theory for certain semilinear rough PDEs was developed in \cite{Gerasimovics-Hairer}. A class of SPDEs amenable to simplification via transformations was studied in \cite{Caruana-Friz-09, Friz-Nilsssen-Stannat}. However, none accommodates the transport noise considered here. The reason is that these approaches rely on the heat semigroup to “restore” the spatial regularity lost due to the multiplicative noise term; in the transport setting, this would require regaining a full derivative, causing the estimates to break down. 

The framework of Unbounded Rough Drivers (URD), introduced in \cite{URDog, EnergymethodRPDEs},  has become the default framework for this type of noise. It has been successfully applied to fluid equations in \cite{HofDeb23, NSpertbyroughtransport} and \cite{Friz-Nilsssen-Stannat, ItoRPDE}. In particular, \cite{NSpertbyroughtransport} introduces a notion of URD solution for a Navier-Stokes equation perturbed by a rough transport noise. 
These solutions are incremental in nature.

By {\it incremental}, we mean that one does not  attempt to define a rough integral directly. Instead, one introduces suitable increment processes $(A_{s,t}^1, A_{s,t}^2)$ (the rough driver), such that for a process $u$ to be a URD solution, the remainder term must be of ``locally" finite $\f 1{3\alpha}$-variation process. In particular, a URD solution only asserts controls on sub-intervals of the time interval for which existence is established. This incremental, partition-based formulation is characteristic of the URD approach. They are reminiscent of Davie's numerical solution \cite{davie} for SDEs, useful but not totally satisfactory. 

We shall define a rough path solution theory. For this, consider 
 \begin{equation}\label{eq:modeleqn}
 du_t =\nu \Delta u_t  - \Pi [(u \cdot \nabla) u] dt-\Pi [(d{\bf Z_t} \cdot \nabla) u],
\end{equation}
where ${\bf Z} = (Z, \mathbb{Z})$ is an $\alpha$-H\"older rough path. We look for a solution satisfying a weak formulation:
 \begin{equ}             \< u_{s,t}, \phi \> = \int_s^t \<\nu \Delta u_r, \phi  \> dr - \int_s^t \< \Pi[(u_r \cdot \nabla) u_r], \phi \> dr 
            - \int_s^t \< \Pi[(d{\bf Z}_r \cdot \nabla )u_r] ,\phi \>.
        \end{equ}
The last integral is understood as a classical rough integral,  the identity holds for any $0\le~s\le~t\le~T$, and the solution is a controlled rough path, eliminating the need of the explicit remainder term in the formulation.
In comparison, a URD solution is only required  to satisfy this identity weakly on sub-intervals of the chosen partition.

\medskip

{\bf Main Results.}

We begin with the analysis of the fast component $w^\epsilon$, defined on $H^\sigma$ for $\sigma> 2 + \frac{d}{2}$. (This choice of $\sigma$ ensures the Sobolev embedding $H^\sigma \hookrightarrow W^{2,\infty}$.) Set  
 \begin{equ}\label{X}
      X_{s,t}^\epsilon=\epsilon^{H-1} \int_s^t  w_r^\epsilon\, dr  \end{equ}
    and denote by $X^{\epsilon, i}$ the projection of $X^{\epsilon}$ onto the $i$th eigenvector of $Q$. We then define
     \begin{equ} \label{XX} \mathbb{X}_{s,t}^{\epsilon,i,j} = \int_s^t X^{\epsilon, i}_{s,r} \otimes 
      dX^{\epsilon, j}_{r}.
     \end{equ} 

The first result establishes the rough functional limit theorem underlying our main claim.

\begin{atheorem}\label{thm:A}
 Assume that $M,Q$ satisfy Assumption \ref{assumption}.   Let  $H\in \left(\frac{1}{3}, \frac{1}{2}\right)$,  $q\in (1, \infty)$,  $\sigma>2+ \frac{3}{2}$, and  $\frac{1}{3}<\alpha<H$.  
    Then $(X^{\epsilon}_{\cdot}, \mathbb{X}^{\epsilon}_{ \cdot\cdot})\to  (M^{-1} Q^{\frac{1}{2}} B^H,\eta)$ in $L^{q}\bigl(\Omega;\mathcal{C}^{\alpha}([0,T];H^{\sigma})
\times \mathcal{C}^{2\alpha}(\Delta_T;H^{\sigma}\otimes H^{\sigma})\bigr)$
where $\eta$ denotes the canonical Gaussian lift.
\end{atheorem}

Our main theorem is the following, proved in Section~\ref{sec:thmB}.
\begin{atheorem} \label{thm:B}
   Assume Assumption \ref{assumption}. Let $\{ u^\epsilon\}$ be a family of rough path solutions to  (\ref{sf}) with initial data $u_0^\epsilon$ satisfying 
    $\sup_{\epsilon>0}\norm{u^\epsilon_0}_{H} <\infty$.   Then there exists a subsequence $\{ u^{\epsilon_n}\}$ such that $u^{\epsilon_n} \rightarrow u$ in  $L^{2}([0,T];H) \cap C([0,T],H_{w})$. 
    
    Furthermore, the limit is a rough path solution to \eqref{eq:limit-eq}
   and the convergence holds in the controlled rough path topology: 
        \begin{equs}
         u^{\epsilon_n} & \rightarrow u  \; \text{ in } C^{\frac{1}{\alpha-\delta} \var}([0,T]; H^{-1});\qquad 
         R^{u^{\epsilon_n}} &\rightarrow R^u, \;  \text{ in } C^{\frac{1}{2(\alpha-\delta)}\var}([0,T]; H^{-2})
        \end{equs}
        for any $\delta>0$.
        \end{atheorem}
     We actually showed that if a rough path $\mathbf Z^\epsilon$ converges to $\mathbf Z$ in the rough topology, the above theorem remains true with \eqref{eq:limit-eq} replaced by \eqref{eq:modeleqn}, see Theorem~\ref{general-ThmB}, concluding that \eqref{eq:modeleqn} has a rough path solution.
      
    To interpret this limit, we introduce the notion of a rough path solution.
  
\begin{definition} \label{def:RPsolution}
  A stochastic process  $u\in L^{2}_TH^1 \cap C_TH_w \cap C^{\frac{1}{\alpha}\var}_TH^{-1}$
    is  a {\it rough path solution} of \eqref{eq:modeleqn} if the remainder
     $R^u_{s,t} := u_{s,t} - \left[ -\Pi \left( Z_{s,t} \cdot \nabla u_s \right) \right]$ belongs to $C^{\frac{1}{2\alpha}\var}_{2,T}H^{-2}$, and if, for all $\phi \in H^3$ and for all $0\le s\le t\le T$,
        \begin{equ} \label{eq:weak_formulation}
            \< u_{s,t}, \phi \> = \int_s^t \<\nu \Delta u_r, \phi  \> dr - \int_s^t \< \Pi[(u_r \cdot \nabla) u_r], \phi \> dr 
            - \int_s^t \< \Pi[(d{\bf Z}_r \cdot \nabla )u_r] ,\phi \>.
        \end{equ}
\end{definition}

A difficulty in this definition is to make sense of the last integral term $\int_s^t \< \Pi[(d{\bf Z}_r \cdot \nabla )u_r] ,\phi \>$.
To identify this rough integral, define
$$y_t:=\< \Pi[(\bullet \cdot \nabla )u_t] ,\phi \>,  \qquad y'_t := \< \Pi \left[ (\bullet \cdot \nabla) \Pi \left[ (\star \cdot \nabla)u_t \right]\right], \phi \>.$$
By Proposition~\ref{rem:distrvstested} and Lemma~\ref{u-is-rough-path},  the pair $(y, y')$ defines a rough path controlled by ${\bf Z}$ and that the integral is defined via the sewing lemma as the limit of enhanced Riemann sums
$$ \lim_{|\pi|\to 0} \sum_{[s', t']\in \pi}  y_{s'}Z_{s',t'} + y'_{s'} \mathbb{Z}_{s',t'}. $$

We now observe that the terms in this sum coincide with
$$ \<A^1_{s,t}u_s, \phi \> = y_s Z_{s,t}, \qquad  \<A^2_{s,t}u_s, \phi \> = y'_s \mathbb{Z}_{s,t},$$
where $A^1$ and $A^2$ are defined from the maps $A^{(1)}$ and $A^{(2)}$ as follows. First set
   $$ A^{(1)}( \cdot, u) : h \mapsto  -\Pi(h \cdot \nabla )u, $$ 
 and 
 $$   A^{(2)}(\cdot,\phi): f \otimes g\mapsto \Pi(g \cdot \nabla \Pi(f \cdot \nabla \phi)).$$
 The operator-valued increments are then given by
\begin{equ}\label{As}
    A^{1}_{s,t} = A^{(1)}(Z_{s,t}, \cdot ), \qquad 
    A^{2}_{s,t} = A^{(2)}(\mathbb{Z}_{s,t}, \cdot ).
\end{equ}

\begin{definition}\label{def:URD}\cite{URDog, EnergymethodRPDEs}
    Let $\alpha\in(\f 13,\f 12]$. An $\alpha$-continuous unbounded rough driver with respect to the scale $(H^{m})_{m>0}$ is a pair $(A^{1},A^{2})$ 
    of two index maps such that for any $s,t\in [0,T]$:
    \begin{align*}
        \norm{A^{1}_{s,t}}_{\mathcal{L}(H^{-m};H^{-(m+1)})} &\lesssim |t-s|^{\alpha}, \qquad m\in [0,2], \\
        \norm{A^{2}_{s,t}}_{\mathcal{L}(H^{-m};H^{-(m+2)})} &\lesssim |t-s|^{2\alpha}, \qquad
        m\in [0,1],
    \end{align*}
    and satisfies the Chen's relations: 
        $\delta A^{1}_{s,r,t}=0$ and $ \delta A^{2}_{s,r,t}=A^{1}_{r,t}A^{1}_{s,r}$  for all $ 0\leq s \leq r \leq t \leq T$.
\end{definition}

\begin{definition} \label{def:URDsoln}
A process $u\in L^{2}_TH^1 \cap C_TH_w$ together with the $\alpha$-continuous unbounded rough driver $(A^1, A^2)$ on the scale $H^r$
 is said to be an {\it Unbounded Rough Driver  (URD) solution} to \eqref{eq:modeleqn}, if 
  there exists a covering $\mathcal P=\{[s_k,t_k]\}$ of $[0,T]$ such that for any~$k$,
        the residual,
        \begin{equation}\label{eq:URDweak_formulation}
            u^\natural_{s,t} := u_{s,t} - \int_{s}^{t} \nu \Delta u_r dr + \int_s^t \Pi(u_r\cdot \nabla u_r) dr - 
             A^1_{s,t}u_s - A^2_{s,t}u_s, \qquad s_k\le s\le t\le t_k,
        \end{equation}
        belongs to the finite variation space $ C_2^{\f 1{3\alpha}\var}([s_k, t_k];H^{-3})$ of two parameter processes.
\end{definition}

This notion of URD solution has been used in e.g. \cite{EnergymethodRPDEs,HofDeb23}. URD solutions are often 
formulated in the terminology of local finite $\frac{1}{3\alpha}$-variation;  which precisely mean that  there exists a covering $\{I_k\}$ of $[0,T]$ such that $u^\natural \in C_2^{\frac{1}{3\alpha}\var}(I_k;H^{-3})$.

The following result is proved in Section~\ref{sec:rough_setting}.

\begin{atheorem} \label{thm:equivalence-of-sol}
Let ${\bf Z} = (Z, \mathbb{Z})$ be an $\alpha$-H\"older rough path on $H^\sigma$. Then $u$ is a rough path solution of \eqref{eq:modeleqn} if and only if it is a URD solution with respect to the driver $(A^1, A^2)$ in \eqref{As}. In particular, this URD solution
enjoys the global property:
$$u^\natural \in C_2^{\f 1{3\alpha}\var}([0,T];H^{-3}).$$
\end{atheorem}

\medskip

{\bf Acknowledgement.}
Li acknowledges support by Swiss National Science Foundation project mint 10000849 and  UKRI  grant
EP/V026100/1.  The authors benefited from the support from Mathematics of Physics NCCR SWISSMAP.

\section{Preliminaries} \label{sec:preliminaries}
We denote
by $\mathcal{L}(H,K)$ the space of bounded linear operators between Hilbert spaces, and by $\L_2(H,K)$ the space of Hilbert-Schmidt 
operators. 
The notation $L^p$ refers to the usual $L^p$ spaces. Similarly, the notation $L^p_TV$ denotes $L^p([0, T]; V)$ and the notation 
$C_TV$ denotes $C([0,T];V)$. Finally, when we write $H_w$ we mean the space $H$ equipped with the weak topology, instead of the norm 
topology.

Let $\Pi: L^2\to L^2_{\Div}$ denote the Leray projection from the $L^2$ space of vector fields to the $L^2$ space of divergence free vector fields.
The Stokes operator $ \nu \Delta \Pi$ generates an $L^2$-based scale of interpolation spaces, defined by 
$H^r=\D((-\Delta \Pi)^r)$ equipped with the graph norm and inner product  $\norm{u}_{H^{r}}=\norm{(-\Delta \Pi)^{\frac{r}{2}}u}$. On the Torus, these spaces are equivalent to the usual Sobolev spaces of divergence free vector fields.

 Let us recall, that
the Navier-Stokes nonlinearity 
$$b(u,v)= -\Pi[(u \cdot \nabla) v] $$ 
 is skew symmetric in the following sense:
$\langle b(u,v),w \rangle = -\langle b(u,w),v \rangle$,
which is obtained by integration by parts, whenever the inner products are well-defined. This yields the identity:
$\langle b(u,v),v \rangle = 0$ crucial for energy estimates.
It satisfies \cite{Boyer-Fabrie-13-book,Sohr-book, Batchelor}: 
$$b\in \L(H^{\alpha_1} \otimes H^{\alpha_2 +1} ,  H^{-\alpha_3}),$$
for any  $\alpha_1,  \alpha_2,  \alpha_3 \geq 0$ satisfying the following condition:
\begin{equation} \label{eq:b_properties}
    \alpha_1+\alpha_2+\alpha_3
    \begin{cases}
        \geq \frac{d}{2}, & \text{whenever } \alpha_i \neq \frac{d}{2} \text{ for any } i,\\
        > \frac{d}{2}, & \text{if } \alpha_i = \frac{d}{2} \text{ for some } i.
    \end{cases}
\end{equation}

\subsection{Controls and rough path spaces}\label{sec:rp}

We work primarily in $p$-variation spaces, and use controls as the main bookkeeping device in a-priori bounds, and occasionally appeal to the H\"older rough path and controlled path viewpoints. This subsection collects the basic tools used repeatedly in Sections~\ref{sec:rough_setting}–\ref{sec:thmB}.

\subsubsection*{Variation spaces and interpolation}

For an interval $I\subset\mathbb{R}$ set
$$
\Delta_I := \{(s,t)\in I^2:\ s\le t\}.
$$
Let $C(\Delta_I;V)$ denote the space of continuous maps $\Delta_I\to V$.

We define the space of continuous two–parameter maps of finite $p$-variation by
$$
C_{2}^{p\var}(I;V)
:= \Big\{A\in C(\Delta_I;V):\ 
\norm{A}_{C_{2}^{p\var}(I;V)}<\infty\Big\},
$$
where
$$
\norm{A}_{C_{2}^{p\var}(I;V)}
:=\sup_{\pi}\Big(\sum_{[u,v]\in\pi}\norm{A_{u,v}}_V^p\Big)^{1/p},
$$
and the supremum runs over partitions $\pi$ of $I$.

A path $g:I\to V$ induces a two–parameter map $\delta g_{s,t}=g_t-g_s$, and
$$
\norm{g}_{C^{p\var}(I;V)}=\norm{\delta g}_{C_{2}^{p\var}(I;V)}.
$$

The following interpolation inequality will be used repeatedly.

\begin{proposition}\cite[Prop.~5.5]{FVroughpaths}
Let $f\in C^{p\var}(I;V)$ and $1\le p<p'<\infty$. Then
$$
\norm{f}_{C^{p'\var}(I;V)}
\le
\norm{f}_{C^{p\var}(I;V)}^{p/p'}
\norm{f}_{\infty}^{1-p/p'}.
$$
\end{proposition}

In particular, for $p\le p'$ one has the continuous embedding
$$
C_{2}^{p\var}(I;V)\hookrightarrow C_{2}^{p'\var}(I;V).
$$

The following compactness fact will be crucial in Section~\ref{sec:thmA}.

\begin{lemma}\label{lem:pvarcptness}
Let $\{f^n\}\subset C(I;V)$ with
$\sup_n\norm{f^n}_{C^{p\var}(I;V)}<\infty$.
If $f^n\to f$ uniformly, then $f^n\to f$ in $C^{p'\var}(I;V)$ for every $p'>p$.
\end{lemma}
We supply a proof for reader's convenience. 
Uniform convergence gives
$$
\norm{f}_{C^{p\var}(I;V)}^p
=\sup_\pi\liminf_{n\to\infty}\sum_{[u,v]\in\pi}\norm{f^n_{u,v}}^p
\le \liminf_{n\to\infty}\norm{f^n}_{C^{p\var}(I;V)}^p.
$$
Applying the interpolation inequality yields
$$
\norm{f^n-f}_{C^{p'\var}(I;V)}
\le
\sup_n\norm{f^n-f}_{C^{p\var}(I;V)}^{p/p'}
\norm{f^n-f}_{\infty}^{1-p/p'}
\to0.
$$

\subsubsection*{Controls}
\begin{definition}\label{def:control}
A map $\omega:\Delta_T\to[0,\infty)$ is called a \emph{control} if it is continuous, $\omega(s,s)=0$, and
$$
\omega(s,t)\ge \omega(s,u)+\omega(u,t)
\qquad\text{for } s\le u\le t.
$$
\end{definition}

Controls are stable under addition and under powers $\omega\mapsto\omega^a$ for $a\ge1$. If $\omega_1,\omega_2$ are controls and $a,b>0$, then $(\omega_1^a\omega_2^b)^{1/(a+b)}$ is also a control.

Typical examples are $\omega(s,t)=|t-s|$ and, for $g\in C^{p\var}(I;V)$,
$$
\omega_g(s,t):=\norm{\delta g}_{C_{2}^{p\var}([s,t];V)}^{p}.
$$

\subsubsection*{Hölder rough paths and controlled paths}

For $\alpha\in(1/3,1/2]$, an $\alpha$-Hölder rough path over $V$ is a pair
$$
(X,\mathbb{X})\in C^\alpha([0,T];V)\times C^{2\alpha}(\Delta_T;V\otimes V)
$$
satisfying Chen’s relation
$$
\delta\mathbb{X}_{s,u,t}
=
\mathbb{X}_{s,t}-\mathbb{X}_{s,u}-\mathbb{X}_{u,t}
=
X_{s,u}\otimes X_{u,t}.
$$
We measure distance using the Hölder norms
$$
\rho_\alpha((X,\mathbb{X}),(Y,\mathbb{Y}))
:=
\norm{X-Y}_{C^\alpha}
+
\norm{\mathbb{X}-\mathbb{Y}}_{C^{2\alpha}},
$$
which becomes a metric after adding the corresponding supremum norms.

\begin{definition}\label{def:controlled-rp}
Let $X\in C^\alpha([0,T];V)$. A pair $(Y,Y')$ with
$Y\in C^\alpha([0,T];W)$ and
$Y'\in C^\alpha([0,T];\mathcal{L}(V,W))$
is said to be \emph{controlled by $X$} if the remainder
$$
R^Y_{s,t}:=Y_{s,t}-Y'_sX_{s,t}
$$
belongs to $C^{2\alpha}(\Delta_T;W)$.
We write $(Y,Y')\in\mathcal{D}_X^{2\alpha}([0,T];W)$ and set
$$
\norm{Y,Y'}_{X,2\alpha}
:=
\norm{Y'}_\alpha+\norm{R^Y}_{2\alpha}.
$$
\end{definition}

Replacing Hölder regularity by $p$-variation yields the analogous space
$$
\mathcal{D}_X^{1/(2\alpha)\var}([0,T];W),
$$
defined in the same way.

\subsection{Gaussian rough path lifts in Hilbert spaces}\label{sec:gaussianrp}
We recall the Gaussian rough path lift criterion from \cite[Chapter 10]{FH}, based on 2-d covariance variation, which will be used in Section~4.

The $2$-d covariance matrix of a real valued Gaussian process $X$ is
$$
R_X(s,t,s',t') := \mathbb{E}\left[X_{s,t} \otimes X_{s't'}\right].
$$
The following norms of the $2$-d covariance will be useful:
\begin{equation}\label{covariance-norm}
\begin{aligned}
    \norm{R_X}_{\rho; I \times I'}
    &:= \left[\sup_{\mathcal{P},\mathcal{P}'}
    \sum_{[u,v]\in \mathcal{P}, [u',v']\in \mathcal{P'}}
    \left|R_X(u,v,u',v')\right|^{\rho}\right]^{\frac{1}{\rho}},\\
    \norm{R_X}_{\rho\text{-var};I \times I'}
    &:= \left[\sup_{\mathcal{Q}}
    \sum_{[u,v]\times[u',v']\in \mathcal{Q}}
    \left|R_X(u,v,u',v')\right|^{\rho}\right]^{\frac{1}{\rho}},
\end{aligned}
\end{equation}
where $\mathcal{P},\mathcal{P}'$ are partitions of $I$ and $I'$ and $\mathcal{Q}$ is a partition of $I\times I'$. Since $\mathcal Q$ contains more elements than grid-like partitions, one has
$$
\norm{R_X}_{\rho;I \times I'} \leq \norm{R_X}_{\rho\text{-var};I \times I'}.
$$
We shall choose $\rho$ so that these quantities are finite for the processes under consideration.

Given independent mean-zero Gaussian processes $\{X^i\}$, we define the second level by the $L^2$-limit of Riemann sums
$$
\mathbb{X}^{i,j}_{s,t}
:= \lim_{|\mathcal{P}|\to 0}\sum_{[u,v] \in\mathcal{P}} X^i_{s,u} X^j_{u,v},
\qquad i<j,
$$
and extend to all $(i,j)$ by the geometric relations. Existence of this lift and the estimates used below follow from \cite[Chapter 10]{FH}.

We consider Gaussian processes on a separable Hilbert space $\mathcal{H}$ with orthonormal basis $\{e_i\}$. We write
$X_t = \sum_i X^i_t e_i$ for its components in this basis and work componentwise.

This result is a direct Hilbert space extension of \cite[Theorem 10.5]{FH}; the proof is identical and therefore omitted.

\begin{theorem}\label{thm:FH10.5}
    Let $(X,Y)=(X^i,Y^i)_i$ be a centered Gaussian process on $[0,T]$, such that $(X^i,Y^i)$ is independent of 
    $(X^j, Y^j)$ when $i \neq j$. Assume that there exists $\rho \in [1,2)$, and $\{M_i\}\in \ell^1$, such that 
    \begin{align*}
        \norm{R_{X^i}}_{\rho,[s,t]^2} &\leq M_i |t-s|^{\frac{1}{\rho}}, \qquad
        \norm{R_{Y^i}}_{\rho,[s,t]^2} \leq M_i |t-s|^{\frac{1}{\rho}}, \quad 
        \norm{R_{X^i-Y^i}}_{\rho,[s,t]^2} &\leq \epsilon^2 M_i |t-s|^{\frac{1}{\rho}},
    \end{align*}
for all $i\geq 1$ and $0\leq s \leq t \leq T.$ Then
    \begin{itemize}
        \item [a)] For every $q\in [1,\infty)$ and for all $0\leq s \leq t \leq T$,
       $$
            \mathbb{E} \left[\norm{Y_{s,t} - X_{s,t}}^q \right]^{\frac{1}{q}} \lesssim \epsilon \sqrt{\norm{M}_{\ell^1}}
            |t-s|^{\frac{1}{2\rho}}, \qquad 
            \mathbb{E} \left[\norm{\mathbb{Y}_{s,t} - \mathbb{X}_{s,t}}^q \right]^{\frac{1}{q}} \lesssim \epsilon \norm{M}_{\ell^1}
            |t-s|^{\frac{1}{\rho}}.
$$
        \item [b)] For any $\alpha < \frac{1}{2\rho}$ and $q\in [1,\infty)$, one has
        \begin{align*}
            \mathbb{E} \left[\norm{Y-X}_\alpha^q\right]^{\frac{1}{q}} &\lesssim \epsilon \sqrt{\norm{M}_{\ell^1}}, \qquad
            \mathbb{E} \left[\norm{\mathbb{Y}-\mathbb{X}}_{2\alpha}^q\right]^{\frac{1}{q}} \lesssim \epsilon \norm{M}_{\ell^1}.
        \end{align*}
        \item[c)] For $\rho \in [1, \frac{3}{2})$ and any $\alpha \in (\frac{1}{3}, \frac{1}{2\rho})$ , $q<\infty$ one has 
       $$ \mathbb{E}\left[\rho_\alpha(\mathbf{X},\mathbf{Y})^q \right]^{\frac{1}{q}} \lesssim \epsilon. $$
    \end{itemize}
\end{theorem}
The previous theorem shows that rough path convergence in the Gaussian case, reduces to controlling the norms of the two–dimensional covariance of the processes and their difference.
The following criterion allows one to control the finite variation norm of the $2$-d covariance \cite{coutin-qian}.
\begin{theorem}[16.5~\cite{FVroughpaths}] \label{thm:CQ} 
    Suppose that for a Gaussian process $Y$, $H\in(0,1)$, $c_H>0$, and all $ 0<h<t-s$, the following holds:
    \begin{equ}\label{Coutin-Qian}
        \mathbb{E} \left[ |Y_{s,t}|^2\right] \leq c_H |t-s|^{2H}, \qquad 
        \left| \mathbb{E} \left( Y_{s, s+h} Y_{t, t+h} \right)\right| \leq c_H |t-s|^{2H-2}h^2, 
    \end{equ}
then the covariance of $Y$ has finite $\frac{1}{2H}$-variation, and
    $$ \norm{R_Y}_{\frac{1}{2H}-var;[s,t]^2} \lesssim  c_H|t-s|^{2H}, \text{ for } 0\leq s\leq t\leq T. $$
\end{theorem}

\section{Rough path solutions and their equivalence with URD solutions} \label{sec:rough_setting}

This section proceeds in three steps: (i) construction of the rough integral appearing in the weak formulation \eqref{eq:weak_formulation}, (ii) analytic bounds for the associated driver and URD solutions, and (iii)  the equivalence of the two notions of solutions.

The proof strategy :
We first show that the operators $A^1$ and $A^2$ associated with ${\bf Z}$ form an unbounded rough driver and that their operator norms are controlled by the $\alpha$-Hölder norm of ${\bf Z}$. 
Subsequently, we introduce controls measuring the size of the driver and the drift terms. The core of the argument is then to exploit a compensation between spatial and temporal regularity, achieved by using different representations for the remainder term, and using smoothing operators. This mechanism yields local estimates for the finite variation norms of $u$, $R^u$, and $u^\natural$. These estimates initially require restricting the length of time intervals, and are subsequently upgraded to global ones using the properties of the controls. 

We also note that the same estimates play an important role in the next section, where they are used to obtain uniform bounds for the homogenisation result. Combined with Theorem~\ref{thm:A}, they provide the compactness input for Theorem~\ref{thm:B}.

\subsection{Construction of the rough integral in the weak formulation}

We begin with a simple observation.

\begin{lemma}\label{u-is-rough-path}
If $u$ is a process such that $R^u_{s,t} := u_{s,t} + \left[ \Pi \left( Z_{s,t} \cdot \nabla u_s \right) \right]$ belongs to $C^{\frac{1}{2\alpha}\var}_{2,T}H^{-2}$, then after testing with $\phi\in H^3$,
 $ \< u_{t}, \phi\> $  is a controlled rough path with derivative $-\< \Pi \left[ (\bullet \cdot \nabla)u_t \right], \phi \>$. 
\end{lemma}
\begin{proof}
By assumption,
\begin{equation*}
    \< u_{s,t}, \phi\> + \< \Pi \left( Z_{s,t} \cdot \nabla u_s \right), \phi \> =  \< R^u_{s,t}, \phi \>,
\end{equation*}
and $\langle R^u_{\cdot,\cdot},\phi\rangle$ has finite $\frac{1}{2\alpha}$-variation, so an element of $C^{\frac{1}{2\alpha}\var}([0,T]; \R)$
. This is exactly the
controlled expansion with derivative $-\langle \Pi[(\bullet\cdot\nabla)u_t],\phi\rangle$.
\end{proof}

To define the rough integral, we recall a sewing Lemma in the finite variation context, c.f. \cite[Prop. 3.1]{ItoRPDE}:
\begin{lemma}[Sewing Lemma] \label{lem:sewing}
    Let a two parameter process $H: \Delta \rightarrow V$ satisfy $$ |\delta H_{s \theta t}| \leq \beta\, \omega(s,t)^a $$
    for some control $\omega$ and constants $a>1, \beta>0$ and $\theta\in (s,t)$.
    Then there exists unique maps
    $I : [0,T] \rightarrow V$ and $ I^\natural : \Delta \rightarrow V$ such that
    $$ \delta I_{s,t} = H_{s,t} + I^\natural_{s,t}, \qquad  |I^\natural_{s,t}| \leq  C\,\beta \, \omega(s,t)^a $$
  where $C$ depends only on $a$.     
\end{lemma}
   Since $a>1$, the sewing lemma yields a canonical primitive $I$ (unique up to its initial value),
and we interpret $I$ as the rough integral associated with $H$.

The choice of the Gubinelli derivative of $y$ is motivated by the rough path chain rule.
   Formally viewing $-\Pi (\bullet \cdot \nabla)u_t$ as a map $F(u_t)$, the corresponding Gubinelli derivative is
    $$DF(u_t)u'_t=-\Pi (\bullet \cdot \nabla)u'_t= \Pi(\bullet \cdot \nabla)
    \Pi(\star \cdot \nabla)u_t.$$

\begin{proposition}\label{rem:distrvstested}
Let $u$ denote a rough path solution of \eqref{eq:modeleqn}. Then for any $\phi \in H^3$, $$y_t:=\< \Pi[(\bullet \cdot \nabla )u_t] ,\phi \>,  \qquad y'_t := \< \Pi \left[ (\bullet \cdot \nabla) \Pi \left[ (\star \cdot \nabla)u_t \right]\right], \phi \>$$ is a rough path controlled by $(Z, \Z)$.
This defines the rough integral:
 $$-\int_s^t \< \Pi[(d{\bf Z}_r \cdot \nabla )u_r] ,\phi \> =\lim_{|\pi|\to 0} \sum_{[s', t']\in \pi}   y_{s'}Z_{s',t'} + y'_{s'}\mathbb{Z}_{s',t'}.$$
    Moreover,
$\phi\mapsto \int_s^t \< \Pi(d{\bf Z}_r \cdot \nabla )u_r ,\phi \>$ is a bounded linear functional on $H^3$, and hence can be seen as an element of $H^{-3}$.
\end{proposition}

\begin{proof}
    {\it Step 1:}  Regularity of $y$ and $y'$. 
    \\
We begin by showing that $y_t$ is of finite $\frac{1}{\alpha}$-variation in time. Indeed, we have for a $\psi \in H^\sigma$
\begin{align*}
    \left| y_t \psi - y_s \psi \right| = & \left| \< \Pi[(\psi \cdot \nabla )u_{s,t}] ,\phi \> \right| 
    \leq \norm{\phi}_{H^2} \norm{(\psi \cdot \nabla )u_{s,t}}_{H^{-2}} \leq \norm{\phi}_{H^2} \norm{\psi}_{H^\sigma} \norm{u_{s,t}}_{H^{-1}}, 
\end{align*}
which gives
\begin{equation*}
    \norm{y_{s,t}}_{\mathcal{L}(H^\sigma; \mathbb{R})} \leq \norm{u_{s,t}}_{H^{-1}} \norm{\phi}_{H^2}
\end{equation*}
and therefore
$$ \norm{y}_{C^{\frac{1}{\alpha}\var}([0,T]; \L(H^\sigma; \mathbb{R}))} \leq
 \norm{u}_{C^{\frac{1}{\alpha}\var}([0,T];H^{-1})}\norm{\phi}_{H^2}.$$
By a similar calculation one can show that $y'_t$ is also of finite $\frac{1}{\alpha}$-variation, and one has
$$ \norm{y'}_{C^{\frac{1}{\alpha}\var}([0,T]; \L(H^\sigma \otimes H^\sigma; \mathbb{R}))} \leq
 \norm{u}_{C^{\frac{1}{\alpha}\var}([0,T];H^{-1})}\norm{\phi}_{H^3}.$$

    {\it Step 2:} Controlled structure of $y$ with respect to $Z$.
\\
Observe that
\begin{equation*}
    R^{y}_{s,t} = y_{s,t} - y'_s Z_{s,t} = -\<\Pi \left[ (\bullet \cdot \nabla) R^u_{s,t} \right], \phi \>.
\end{equation*}
Since $R^u_{s,t} := u_{s,t} - \left[ -\Pi \left( Z_{s,t} \cdot \nabla u_s \right) \right]$ belongs to
 $C^{\frac{1}{2\alpha}\var}_{2,T}H^{-2}$ by the definition,  for $\phi \in H^3$, 
 $$ \norm{ R^{y}_{s,t}}_{\L(H^\sigma;\R)} \leq \norm{R^u_{s,t}}_{H^{-2}} \norm{\phi}_{H^3}.$$
Consequently,  $y\in \mathcal{D}^{\f 1 {2\alpha} \var}_X([0,T]; H^{-2})$. 

    {\it Step 3:} Application of the sewing lemma.
\\
We define the increment process 
$$ H_{s,t} = y_s Z_{s,t} +y'_s \mathbb{Z}_{s,t}, $$
 verify the conditions of the sewing lemma,  and show that there is a rough integral  $\int (y, y')_s d\Z_s$. Note that $ H_{s,t} =\<A_{s,t}^1 u_s, \phi\>+\<A_{s,t}^2 u_s, \phi\>$.
Expanding $\delta H$ using Chen's relations:
\begin{align} \label{eq:increment-est}
    \left| \delta H_{s,u,t} \right| &= \left|y_s Z_{s,t} - y_u Z_{u,t} -y_s Z_{s,u} +
   y'_s \mathbb{Z}_{s,t} -y'_u \mathbb{Z}_{u,t} -y'_s \mathbb{Z}_{s,u} \right|\\
    &= \left| -y_{s,u} Z_{u,t} + y'_s Z_{s,u}\otimes Z_{u,t} - 
   y'_{s,u}\mathbb{Z}_{u,t} \right|\\
    &= \left| -R^{y}_{s,u} Z_{u,t} - y'_{s,u}\mathbb{Z}_{u,t} \right|\\
    &\leq \norm{R^u_{s,u}}_{H^{-2}} \norm{\phi}_{H^3} \norm{Z_{u,t}}_{H^\sigma} + \norm{y'_{s,u}}_
    {\L(H^\sigma \otimes H^\sigma; \R)}\norm{\mathbb{Z}_{u,t}}_{H^\sigma \otimes H^\sigma}\\
    &\lesssim \omega(s,t)^{3\alpha}\norm{\phi}_{H^3},
\end{align}
where $\omega$ defined below is a control, as the product of controls with exponents adding to $\geq1$ is again a control
$$\omega(s,t) = \left( \norm{R^u}_{C^{\frac{1}{2\alpha}\var}([s,t];H^{-2})}^{\frac{1}{2\alpha}}\right)^{\frac{2}{3}}|t-s|^{\frac{1}{3}} + 
\left(\norm{u}_{C^{\frac{1}{\alpha}\var}([s,t];H^{-1})}^{\frac{1}{\alpha}}\right)^{\frac{1}{3}}|t-s|^{\frac{2}{3}}.$$

    {\it Step 4:} Identification as a bounded linear functional.\\
With these estimates, the sewing Lemma \ref{lem:sewing} applies showing that the rough integral is defined as claimed and $\phi\mapsto \int_s^t \< \Pi(d{\bf Z}_r \cdot \nabla )u_r ,\phi \>$ is a bounded linear functional on $H^3$. Linearity of the increment process in $\phi$ is a consequence of the uniqueness of the limiting linear functional from the sewing lemma, while boundedness follows from \eqref{eq:increment-est} and the quantitative bound in the sewing Lemma.
\end{proof}

\subsection{Unbounded rough drivers associated with fractional transport noise}
\label{sec:driver-regularity}

To understand why the operators $A^1,A^2$ arise naturally in this setting, it is instructive to consider the equation when $Z$ is smooth. In this case one can iterate the equation and identify the structure that will later be encoded by the unbounded rough driver.

Let $u$ be the solution to \eqref{eq:modeleqn} with smooth $Z$. Write
\begin{equ}\label{mu}
u_{s,t}=\mu_{s,t}- \int_s^t \Pi[(\dot{Z}_r \cdot \nabla) u_r ] dr, \qquad \mu_{s,t} := \int_s^t (\nu  \Delta u_r -\Pi[(u_r\cdot \nabla) u_r] )dr.
\end{equ} 
 Iterating this identity twice yields:
\begin{equs} \label{eqn:iteration}
  &  u_{s,t} -\mu_{s,t}=- \int_s^t \Pi[(\dot{Z}_r \cdot \nabla) u_r ] dr\\
    &=-  \int_s^t \Pi \left[ (\dot{Z}_r \cdot \nabla)u_s \right] dr - \int _s^t  \Pi \left[ (\dot{Z}_r \cdot \nabla)\mu_{s,r} \right] dr +
    \int_s^t \Pi \left[ (\dot{Z}_r \cdot \nabla) \int_s^r \Pi \left[ (\dot{Z}_\theta \cdot \nabla) u_\theta \right] 
    d\theta\right] dr\\
    &= -  \Pi [(Z_{s,t}\cdot \nabla)u_s ] +
    \int_s^t \Pi \left[ (\dot{Z}_r \cdot \nabla) \Pi \left[ (\star\cdot \nabla) u_s  (Z_{s,r})\right] 
    d\theta\right] dr + u^\natural_{s,t}\\
    &= - \Pi \left[ (\bullet \cdot \nabla)u_s \right]Z_{s,t}  +
    \Pi \left[ (\bullet \cdot \nabla) \Pi \left[ (\star \cdot \nabla)u_s\right] \right] \mathbb{Z}_{s,t} + u^\natural_{s,t},
\end{equs}
where in the penultimate line we replaced $u_\theta$ by $u_s$ in the last term, and we defined the remainder
\begin{align*} 
    u^{\natural}_{s,t} &= -\int_{s}^{t}\Pi [(\dot{Z}_r \cdot \nabla) \mu_{s,r}] dr
    + \int_{s}^{t}\int_{s}^{r} \Pi[(\dot{Z}_r \cdot \nabla) \Pi[(\dot{Z}_\theta \cdot \nabla) \mu_{s,r}]] d\theta dr\\
    & \qquad - \int_{s}^{t}\int_{s}^{r}\int_{s}^{\theta} \Pi [(\dot{Z}_r \cdot \nabla) \Pi[(\dot{Z}_\theta \cdot \nabla)
    \Pi [ (\dot{Z}_v \cdot \nabla) u_{v}]]] dv d\theta dr.
\end{align*}
This computation is typical of rough path analysis and is closely related to Davie’s interpretation of rough differential equations via difference equations \cite{davie}. It shows that, even in the smooth case, the increments of the solution naturally decompose into a term linear in $Z_{s,t}$, a term linear in $\Z_{s,t}$, and a higher-order remainder.

This observation motivates the introduction of the operators $A^{(1)}, A^{(2)}$ and suggests that, in the rough setting, the correct notion of solution should require the remainder to satisfy
$$u^\natural \in C^{p\var}_T H^{-3}.$$

We now recall the unbounded operators, introduced in the introduction, and describe their analytic properties.
For a vector field $u$, define 
  $$ A^{(1)}( \cdot, u) : h \mapsto  -\Pi(h \cdot \nabla )u. $$ 
When $u\in C^{\frac{1}{\alpha}\var}([0,T];H^{-1})$, the Sobolev embedding $H^\sigma \hookrightarrow W^{2, \infty}$ implies that $A^{(1)}(\cdot, u)$ takes values in $H^{-2}$.
Its iterated action defines :
$$ A^{(2)}(\cdot,\phi): f \otimes g \mapsto \Pi(g \cdot \nabla \Pi(f \cdot \nabla \phi)).
$$
We then set
\begin{equ}
    A^{1}_{s,t} = A^{(1)}(Z_{s,t}, \cdot ), \qquad 
    A^{2}_{s,t} = A^{(2)}(\mathbb{Z}_{s,t}, \cdot ).
\end{equ}
The relevance of these spaces is that the solution $u$ has better temporal regularity when viewed lower in the Sobolev scale, which is precisely the level at which the rough integral in the weak formulation is interpreted.

\begin{lemma} \label{lem:URDlinktoRP}
    Let $(Z,\mathbb{Z})$ be an $\alpha$-H\"older continuous rough path with values in $H^\sigma$. Then   
      \begin{align*}
        \norm{A^1_{s,t}}_{\L(H^{-m}, H^{-(m+1)})} &\lesssim \norm{ Z_{s,t}}_{H^{\sigma}} \quad \text{for $m\in[0,2]$},\\
        \norm{A^2_{s,t}}_{\L(H^{-m}, H^{-(m+2)})} &\lesssim \norm{ \mathbb{Z}_{s,t}}_{H^{\sigma}\otimes H^\sigma} \quad \text{for $m\in[0,1]$}.
    \end{align*}
Moreover, $(A^{1}, A^{2}) $ is an unbounded rough driver.
\end{lemma}

\begin{proof}
Let $0 \leq s \leq r \leq t \leq T$. The Chen's relations
  $$\delta A^{1}_{s,r,t}=0 \qquad  \delta A^{2}_{s,r,t}=A^{1}_{r,t}A^{1}_{s,r}$$
follow directly from the Chen's relations for $\mathbb{Z}$,
 $$\Z_{s,t} - \Z_{s,u}-\Z_{u,t} = Z_{s,u} \otimes Z_{u,t},$$
 and bilinearity of $A^{(2)}$. Indeed, for any $\phi \in H^{-m}$, 
 \begin{align*}
    \delta A^{2}_{s,u,t}\phi &:= A^{2}_{s,t}\phi -A^{2}_{s,u}\phi -A^{2}_{u,t}\phi \\
    &= A^{(2)}(\mathbb{Z}_{s,t},\phi) - A^{(2)}(\mathbb{Z}_{s,u},\phi) - A^{(2)}(\mathbb{Z}_{u,t},\phi) \\
    &= A^{(2)}(Z_{s,u}\otimes Z_{u,t}, \phi)=A^{(1)}(Z_{u,t}, A^{(1)}(Z_{s,u},\phi) \, )\\
    &= A^{1}_{u,t}A^{1}_{s,u}\phi.
\end{align*}
For the analytic bounds, assume that  $\norm{\phi}_{H^{-m}}\leq 1$.  

{\it Estimate for $A^1$.} Use skew symmetry of$v\mapsto -\Pi[(u \cdot \nabla) v] $, the divergence free property of  $Z_{s,t}$, and integration by parts
\begin{align*}
    \norm{A^{1}_{s,t}\phi}_{H^{-(m+1)}} &= \sup_{\norm{\psi}_{H^{m+1}} \leq 1} \langle Z_{s,t} \cdot \nabla\phi, \psi \rangle
    = \sup_{\norm{\psi}_{H^{m+1}} \leq 1} \langle Z_{s,t} \cdot \nabla\psi, \phi \rangle \\
    &\leq  \sup_{\norm{\psi}_{H^{m+1}} \leq 1} \norm{ Z_{s,t} \cdot \nabla\psi}_{H^{m}} \leq \norm{ Z_{s,t}}_{W^{m,\infty}}
    \lesssim \norm{ Z_{s,t}}_{H^{\sigma}}
\end{align*}
since $H^{\sigma} \subset W^{m, \infty}$ for $\sigma>2+\f 32$, where $m \in [0,2]$ . 

{\it Estimate for $A^2$.}  Using divergence–free structure and integration by parts,
\begin{equation*}
    \langle A^{2}_{s,t}\phi , \psi \rangle = \langle A^{(2)}(\mathbb{Z}_{s,t}, \phi), \psi \rangle
     = \langle \phi, A^{(2)}(\mathbb{Z}_{s,t}^T, \psi) \rangle
\end{equation*}
where $(x\otimes y)^T=y\otimes x$. Hence for $m\in [0,1]$ and $\phi \in H^{-m}$ with $\norm{\phi}_{H^{-m}} \leq 1$:
\begin{align*}
    \norm{A^{2}_{s,t}\phi}_{H^{-(m+2)}} &= \sup_{\norm{\psi}_{H^{m+2}} \leq 1} \langle \phi,A^{(2)}(\mathbb{Z}_{s,t}^T, \psi) \rangle 
     \leq \sup_{\norm{\psi}_{H^{m+2}} \leq 1} \norm{A^{(2)}(\mathbb{Z}_{s,t}^T, \psi)}_{H^{m}} \lesssim \norm{\mathbb{Z}_{s,t}}_{H^{\sigma} \otimes
     H^{\sigma}},
\end{align*}
again using $H^{\sigma}\hookrightarrow W^{m+1,\infty}$.
\end{proof}

Assuming that ${\bf Z}$ is an $\alpha$-H\"older rough path, by Lemma \ref{lem:URDlinktoRP} we obtain:
\begin{corollary} \label{cor:K}
Let $C$ be the proportionality constant from Lemma~\ref{lem:URDlinktoRP} and define
$$K= C\Bigl(\norm{Z}^{\frac{1}{\alpha}}_{C^{\alpha}([0,T];H^{\sigma})}+ 
    \norm{\mathbb{Z}}^{\frac{1}{2\alpha}}_{C^{2\alpha}([0,T];H^{\sigma} \otimes H^{\sigma}}\Bigr).$$
Then
\begin{equ}\label{A-control}
    \norm{A^{1}_{s,t}}^{\frac{1}{\alpha}}_{\mathcal{L}(H^{-m};H^{-(m+1)})} +
    \norm{A^{2}_{s,t}}^{\frac{1}{2\alpha}}_{\mathcal{L}(H^{-m};H^{-(m+2)})}
    \leq  K|t-s|.
\end{equ}
\end{corollary}    

\subsection{A-priori estimates for URD solutions}

We now convert the operator bounds into a control and begin the a-priori estimates. We first derive local estimates for the residual, the solution, and the remainder, which will subsequently be globalized.

Define the drift $\mu$ and two basic controls, in the sense of Definition~\ref{def:control},
 \begin{equ}\label{control:roughpath}
    \mu_{s,t} := \int_s^t \nu \Delta u_r - \Pi [(u_r \cdot \nabla) u_r] dr, \quad \omega_\mu(s,t) := \int_{s}^{t} (1+\norm{u_{r}}_{H^{1}})^{2}dr, \quad \omega_{A}(s,t) := K|t-s|,
  \end{equ}
  where $K$ is the constant from \eqref{A-control}, which depends only on the rough H\"older norm of $\Z$.
By definition, 
\begin{equation}\label{omega-A}
 \norm{A^{1}_{s,t}}^{\frac{1}{\alpha}}_{\mathcal{L}(H^{-n};H^{-(n+1)})} +
    \norm{A^{2}_{s,t}}^{\frac{1}{2\alpha}}_{\mathcal{L}(H^{-n};H^{-(n+2)})} \leq
    \omega_{A}(s,t),
\end{equation}
and 
\begin{equ}\label{control-1}
\begin{aligned}
    \norm{\mu_{s,t}}_{H^{-1}} \lesssim \int_s^t \norm{\Delta u_r}_{H^{-1}} + \norm{(u_r \cdot \nabla) u_r}_{H^{-1}} dr \lesssim  \int_s^t (1+\norm{u_{r}}_{H^{1}})^{2}dr =:
     \omega_\mu(s,t).
     \end{aligned}
\end{equ}

We introduce a family of smoothing operators on the Sobolev scale $(H^m)$.

\begin{definition}\label{smoothing-est}
    A smoothing on a scale of Banach spaces $(H^{m},\norm{\cdot}_{m})_{m\in \mathbb{R}_{+}}$ is a family of operators
    $(J^{\eta})_{\eta \in (0,1]}$, approximating the identity such that for all $\eta \in (0,1]$:
    \begin{equ}\label{frequency}
        \norm{J^\eta-I}_{\mathcal{L}(H^l;H^j)}\leq C\eta^{l-j}  \hbox{ for }  l>j, \qquad 
        \norm{J^\eta}_{\mathcal{L}(H^j;H^l)}\leq C\eta^{-(l-j)} \hbox{ for }  l \geq j.
    \end{equ}
\end{definition}
\begin{remark}
On the Sobolev scale, such smoothing operators can be constructed using Fourier frequency cutoffs; see \cite{HofDeb23, NSpertbyroughtransport}.
\end{remark}

We are now ready to state our first a-priori estimate on URD solutions.
\begin{proposition}[Residual estimate] \label{prop:residualest}
    Let $u$ be a URD solution to \eqref{eq:modeleqn}. 
   Then there exists a constant $c(\alpha, T)>0$, such that, for $|s-t|\le L$ with $L=\frac{c}{K}$,
    the residual term defined in (\ref{eq:URDweak_formulation}) satisfies
    \begin{equation} \label{est:residual}
   \omega_\natural(s,t):=  \norm{u^{\natural}}^{\frac{1}{3\alpha}}_{C^{\frac{1}{3\alpha}\var}([s,t];H^{-3})} \lesssim \norm{u}_{L^{\infty}_{T}H}^{\frac{1}{3\alpha}}\omega_A(s,t) + \omega_A(s,t)^{\f 13}\omega_\mu(s,t)^{\frac{1}{3\alpha}}.
    \end{equation}
    Without loss of generality, we assume $c(\alpha, T)\le 1$.
\end{proposition}

\begin{proof}
The strategy is to exploit the two equivalent representations of $R^u_{s,r}$
  in \eqref{usharpexpress}, trading spatial regularity for temporal regularity via a smoothing argument.
  
  Applying the $\delta$ operator (recall the notation in Section~\ref{sec:rp}) to the equation
    \begin{equation}
            u^\natural_{s,t} := u_{s,t} - \int_s^t \nu \Delta u_r dr + \int_s^t \Pi(u_r\cdot \nabla u_r) dr - 
             A^1_{s,t}u_s - A^2_{s,t}u_s,
\end{equation}
 and exploiting the Chen's relations, we arrive at:
    \begin{equation}
        \< \delta u^{\natural}_{s,r,t},\phi\> = \<  u_{s,r}, A^{2,*}_{r,t} \phi \> + 
        \< R^u_{s,r} ,A^{1,*}_{r,t}\phi \>,
    \end{equation}
    where $A^{i,*}$ refers to the densely defined extension of the adjoint of the differential operator obtained using integration by parts.
    We note, that $R^u$ can be written in two equivalent ways:
    \begin{equation}
        \label{usharpexpress}
        R^u_{s,r} =  u_{s,r} - A^{1}_{s,r}u_{s}= \mu_{s,r} +
        A^{2}_{s,r}u_{s} + u_{s,r}^{\natural}.
    \end{equation}
    The key observation is to note that the first equality in \eqref{usharpexpress} has more spatial regularity, while the second has better 
    regularity in time. 
 To exploit this property, we decompose the test function using a smoothing operator and separating high–frequency and low–frequency contributions.
  
  Let $J^{\eta}$ denote a family of smoothing operators. For $\phi \in H^{3}$, we have
    \begin{equation*}
        |\delta u_{s,r,t}^{\natural}(\phi)| \leq
            |\delta u_{s,r,t}^{\natural}((Id-J^{\eta})\phi)|+
            |\delta u_{s,r,t}^{\natural}(J^{\eta}\phi)|. 
    \end{equation*} 
      The first term uses the representation $R^u_{s,r}=u_{s,r} - A^{1}_{s,r}u_{s}$, which has better spatial regularity.
The second term uses $ R^u_{s,r} = \mu_{s,r} +
        A^{2}_{s,r}u_{s} + u_{s,r}^{\natural}$, which has better temporal structure.

  {\bf A.}   We estimate the first term by replacing $R^u_{s,r}$ with $ u_{s,r} - A^{1}_{s,r}u_{s}$, and applying the bounds on \eqref{omega-A}, as well as the smoothing properties of $J^{\eta}$:
    \begin{align*}
        |\< \delta u_{s,r,t}^{\natural}, (Id-J^{\eta})\phi\>| &\lesssim |\<  u_{s,r}, A^{2,*}_{r,t}(Id-J^{\eta})\phi\>| \\
        &\qquad  + |\<  u_{s,r}, A^{1,*}_{r,t}(Id-J^{\eta})\phi\>| +
        |\< u_{s}, A^{1,*}_{s,r}A^{1,*}_{r,t}(Id-J^{\eta})\phi\>| \\
        & \lesssim \norm{u}_{L^{\infty}_{T}H} [\omega_{A}^{2\alpha}(s,t)\norm{(Id-J^{\eta})\phi}_{H^{2}} +  
        \omega_{A}^{\alpha}(s,t)\norm{(Id-J^{\eta})\phi}_{H^1}]\\
        &\lesssim \norm{u}_{L^{\infty}_{T}H} \norm{\phi}_{H^{3}} [\omega_{A}^{2\alpha}(s,t) \eta + \omega_{A}^{\alpha}(s,t) \eta^{2}].
    \end{align*}
We have applied the high frequency estimate in \eqref{frequency}.

{\bf B.}    We next estimate the second term $ |\delta u_{s,r,t}^{\natural}(J^{\eta}\phi)|$ by first expanding it with the expression $ R^u_{s,r} = \mu_{s,r} +
        A^{2}_{s,r}u_{s} + u_{s,r}^{\natural}$, then expand $u_{s,r}$ :  
    \begin{align*}
        |\< \delta u_{s,r,t}^{\natural}, J^{\eta}\phi \>| &\lesssim 
        | \< u_{s,r}, A^{2,*}_{r,t}J^{\eta}\phi\>| + 
        |\< \mu_{s,r}, A^{1,*}_{r,t}J^{\eta}\phi\>| + 
        |\< A^{2}_{s,r}u_{s}, A^{1,*}_{r,t}J^{\eta}\phi\>| + 
        |\< u_{s,r}^{\natural},A^{1,*}_{r,t}J^{\eta}\phi\>| \\
        &\lesssim |\<\mu_{s,r},A^{2,*}_{r,t}J^{\eta}\phi\>| +
        | \<u_{s},A^{1,*}_{s,r}A^{2,*}_{r,t}J^{\eta}\phi\>| +
        | \<u_{s},A^{2,*}_{s,r}A^{2,*}_{r,t}J^{\eta}\phi\>| +
        |\< u_{s,r}^{\natural}, A^{2,*}_{r,t}J^{\eta}\phi\>|\\
       &\quad +
         |\<  \mu_{s,r},A^{1,*}_{r,t}J^{\eta}\phi\>| + 
         |\< u_{s},A^{2,*}_{s,r} A^{1,*}_{r,t}J^{\eta}\phi\>| + 
         |\<u_{s,r}^{\natural},A^{1,*}_{r,t}J^{\eta}\phi\>|.
             \end{align*}
          We then apply the control estimates (\ref{omega-A}-\ref{control-1}) on $A^1,A^2$ and $\mu$ to obtain
        \begin{align*}
        & \lesssim \omega_{\mu}(s,t) \omega_{A}^{2\alpha}(s,t) \norm{J^{\eta}\phi}_{H^{3}} +
        \norm{u}_{L^{\infty}_{T}H}\omega_{A}^{3\alpha}(s,t)\norm{J^{\eta}\phi}_{H^{3}} +
        \norm{u}_{L^{\infty}_{T}H}\omega_{A}^{4\alpha}(s,t)\norm{J^{\eta}\phi}_{H^{4}} \\
        & \qquad +\omega_{A}^{2\alpha}(s,t)\omega_{\natural}^{3\alpha}(s,t)\norm{J^{\eta}\phi}_{H^{5}}  + 
        \omega_{\mu}(s,t)\omega_{A}^{\alpha}(s,t) \norm{J^{\eta}\phi}_{H^{2}} +
        \norm{u}_{L^{\infty}_{T}H}\omega_{A}^{3\alpha}(s,t)\norm{J^{\eta}\phi}_{H^{3}} \\
        &\qquad + \omega_{A}^{\alpha}(s,t)\omega_{\natural}^{3\alpha}(s,t)\norm{J^{\eta}\phi}_{H^{4}}
    \end{align*}
  Let $\eta\le 1$.  We apply the smoothing properties of $J^{\eta}$, i.e. the low frequency estimate from \eqref{frequency}, to obtain:
    \begin{align*}
        &\lesssim \bigl[\omega_{\mu}(s,t) \omega_{A}^{2\alpha}(s,t) +
        \norm{u}_{L^{\infty}_{T}H}\omega_{A}^{3\alpha}(s,t)+
        \norm{u}_{L^{\infty}_{T}H}\omega_{A}^{4\alpha}(s,t)\eta^{-1} \bigr]\norm{\phi}_{H^{3}}\\ & \qquad  +
      \bigl[ \omega_{A}^{2\alpha}(s,t)\omega_{\natural}^{3\alpha}(s,t)\eta^{-2} 
        +\omega_{\mu}(s,t)\omega_{A}^{\alpha}(s,t) +
        \norm{u}_{L^{\infty}_{T}H}\omega_{A}^{3\alpha}(s,t) \bigr]\norm{\phi}_{H^{3}}\\ & \qquad + 
       \bigl[ \omega_{A}^{\alpha}(s,t)\omega_{\natural}^{3\alpha}(s,t) \eta^{-1}\bigr]\norm{\phi}_{H^{3}}.
    \end{align*}
    
    We now optimize the choice of the smoothing parameter. Set $\eta=\omega_{A}^{\alpha}(s,t)\lambda$, where $\lambda>0$ to be determined later, and $s,t$ are such that  $|s-t|\leq \frac{1}{\lambda^{\frac{1}{\alpha}}K}=:L$ so that $\eta\leq 1$.

  Combine the two estimates to obtain:
    \begin{align*}
        \norm{\delta u_{s,r,t}^{\natural}}_{H^{-3}} \lesssim
       & \norm{u}_{L^{\infty}_{T}H} \omega_{A}^{3\alpha}(s,t)[1+\lambda+\lambda^{2}+\lambda^{-1}]  
      +  \omega_{\mu}(s,t)\omega_{A}^{2\alpha}(s,t)\\
    &\qquad    + \omega_{\mu}(s,t)\omega_{A}^{\alpha}(s,t) +
        \omega_{\natural}^{3\alpha}(s,t)[\lambda^{-1}+\lambda^{-2}].
    \end{align*}
    Raising to $1/3\alpha$ power for $\alpha<\f 13$, and use sub-additivity, we have:
    \begin{align*}
        \norm{\delta u_{s,r,t}^{\natural}}_{H^{-3}} ^{\frac{1}{3\alpha}} &\lesssim
        \norm{u}_{L^{\infty}_{T}H}^{\frac{1}{3\alpha}} \omega_{A}(s,t)[1+\lambda+\lambda^{2}+\lambda^{-1}]^{\frac{1}{3\alpha}}  +
        \omega_{\mu}(s,t)^{\frac{1}{3\alpha}}\omega_{A}^{\f 23}(s,t)\\
       &\qquad + \omega_{\mu}^{\frac{1}{3\alpha}}(s,t)\omega_{A}^{\f 13}(s,t) +
        \omega_{\natural}(s,t)[\lambda^{-1}+\lambda^{-2}]^{\frac{1}{3\alpha}}.
    \end{align*}
    As $\alpha \in (\f 13,\f 12)$, we have $\frac{1}{3\alpha} \in(\f 23,1)$, so the right hand side defines a control. Hence by the sewing Lemma
    (using that $u^\natural$ is of finite $\frac{1}{3\alpha}$-variation):
    \begin{align*}
        \norm{ u_{s,t}^{\natural}}_{H^{-3}} ^{\frac{1}{3\alpha}} 
        &\le C \norm{u}_{L^{\infty}_{T}H}^{\frac{1}{3\alpha}}
        \omega_{A}(s,t)[1+\lambda+\lambda^{2}+C\lambda^{-1}]^{\frac{1}{3\alpha}}  +
        \omega_{\mu}^{\frac{1}{3\alpha}}(s,t)\omega_{A}^{\f 23}(s,t)\\
     &\quad +  C \omega_{\mu}(s,t)^{\frac{1}{3\alpha}}\omega_{A}^{\f 13}(s,t) +
       C \omega_{\natural}(s,t)[\lambda^{-1}+\lambda^{-2}]^{\frac{1}{3\alpha}}.
    \end{align*}
   Recall the notation $  \omega_\natural(s,t):=  \norm{u^{\natural}}^{\frac{1}{3\alpha}}_{C^{\frac{1}{3\alpha}\var}([s,t];H^{-3})}$.
  We now choose $\lambda>0$ such that: $$C[\lambda^{-1}+\lambda^{-2}]^{\frac{1}{3\alpha}}\leq\frac{1}{2},$$ where $C$ depends
     only on $T,\alpha$, and the smoothing implied constants.  Crucially this does not depend on $u$ or the rough path defining the driver in
    any way.

     This choice absorbs the contribution of $\omega_{\natural}$, on the right-hand side, closing the estimate.
       Employ the variation norm, by taking supremum over all partitions of $[s,t]$,  to conclude:
    \begin{equation*}
        \omega_{\natural}(s,t) \le 2C\norm{u}_{L^{\infty}_{T}H}^{\frac{1}{3\alpha}} \omega_{A}(s,t) +
     2C   \omega_{\mu}^{\frac{1}{3\alpha}}(s,t)\omega_{A}^{\f 23}(s,t) +
       2C \omega_{\mu}^{\frac{1}{3\alpha}}(s,t)\omega_{A}^{\f 13}(s,t),
    \end{equation*}
    and hence also:
    \begin{equation*}
        \omega_{\natural}(s,t) \lesssim \norm{u}_{L^{\infty}_{T}H}^{\frac{1}{3\alpha}} \omega_{A}(s,t) +
        \omega_{\mu}^{\frac{1}{3\alpha}}(s,t)\omega_{A}^{\f 13}(s,t)
    \end{equation*}
    holds for all $|s-t|\leq \frac{1}{\lambda^{\frac{1}{\alpha}}K}=:L$.
\end{proof}

Having obtained the estimate on the remainder, we will now obtain a similar estimate on the solution. Define
\begin{equ}\label{eq:Ltilde}
        \tilde{L} := (C+1)\left( K^\alpha T^{2\alpha/3} + \norm{u}_{L^\infty_TH}^{\f 13}K^{\alpha/3}\omega_\mu^{\f 13}(0,T) \right)^{-\frac{3}{\alpha}} \wedge L,
    \end{equ}
    where $C$ is the constant ignored in \eqref{est:residual}.
\begin{proposition}[solution estimate] \label{prop:solutionest}
    Let $u$ be a URD solution to \eqref{eq:modeleqn}. For any~$|t-s|<\tilde{L}$:
    \begin{equation} \label{est:u}
   \omega_u(s,t):= \norm{u}_{C^{\frac{1}{\alpha}\var}([s,t];H^{-1})}^{\frac{1}{\alpha}}  \lesssim (1+\norm{u}_{L_T^\infty H})^{\frac{1}{\alpha}}\left[ 
            \omega_\mu^{\frac{1}{\alpha}}(s,t) + \omega_{\mu}^{\frac{1}{3\alpha}}(s,t)\omega_{A}^{\f 13}(s,t) \right].
    \end{equation}
\end{proposition}
\begin{proof}     Let $\phi \in H^{1}$. Splitting $|\<u_{s,t},\phi\>|$ for $\phi \in H^{1}$ with the smoothing operators  $J^{\eta}$,
the high frequency part is controlled simply by
$$    |\<  u_{s,t},(Id-J^{\eta})\phi\>| \le  \norm{u}_{L^{\infty}_{T}H}\norm{(Id-J^{\eta})\phi}_{H}
\lesssim \norm{u}_{L^{\infty}_{T}H}\eta \norm{\phi}_{H^{1}}.$$
To control the low frequency part, apply the rough formulation of the equation on $[s,t]$ and estimate with the controls:
    \begin{align*}
     |\<  u_{s,t},J^{\eta}\phi\>|&\le  \omega_{\mu}(s,t)\norm{J^{\eta}\phi}_{H^{1}}+
        \omega_{A}^{\alpha}(s,t)\norm{u}_{L^{\infty}_{T}H}\norm{J^{\eta}\phi}_{H^{1}} \\
        &\quad +\omega_{A}^{2\alpha}(s,t)\norm{u}_{L^{\infty}_{T}H}\norm{J^{\eta}\phi}_{H^{2}} +
        \omega_{\natural}^{3\alpha}(s,t)\norm{J^{\eta}\phi}_{H^{3}}.
     \end{align*}
For any $\eta\le 1$,  these yield   \begin{align*}
        |\<  u_{s,t},\phi\>| &\leq
        |\<  u_{s,t}, (Id-J^{\eta})\phi\>|+
        |\<  u_{s,t},J^{\eta}\phi\>|  \\
        & \lesssim [\norm{u}_{L^{\infty}_{T}H}\eta +
        \omega_{\mu}(s,t)+ 
        \omega_{A}^{\alpha}(s,t)\norm{u}_{L^{\infty}_{T}H}+
        \omega_{A}^{2\alpha}(s,t)\norm{u}_{L^{\infty}_{T}H}\eta^{-1} + \omega_{\natural}^{3\alpha}(s,t)\eta^{-2}]
        \norm{\phi}_{H^{1}}
    \end{align*}
   Set  $\eta= \omega_{A}^{\alpha}(s,t)+\omega_{\natural}^{\alpha}(s,t)$. Taking $[s,t]$ small so that $\eta\le 1$  yields
    \begin{align*}
        \norm{ u_{s,t}}_{H^{-1}} &\lesssim (1+\norm{u}_{L^{\infty}_{T}H})[
        \omega_{\mu}(s,t) + \omega_{A}^{\alpha}(s,t)+
        \omega_{\natural}^{\alpha}(s,t)].
    \end{align*}
    Now exploiting the estimate \eqref{est:residual}, we obtain the required estimate \eqref{est:u}.
  Let us examine this smallness condition more closely, assuming $|t-s|<L$ from Proposition~\ref{prop:residualest}:
    \begin{align*}
        \eta&= \omega_{A}^{\alpha}(s,t)+\omega_{\natural}^{\alpha}(s,t) \lesssim K^\alpha|t-s|^\alpha + 
        C \left( \norm{u}_{L^\infty_TH}^{\f 13}\omega_A^{\alpha}(s,t) + \omega_A^{\alpha/3}(s,t)\omega_\mu^{\f 13}(s,t) \right)\\
        &\leq (C+1)\left( K^\alpha T^{2\alpha/3} + \norm{u}_{L^\infty_TH}^{\f 13}K^{\alpha/3}\omega_\mu^{\f 13}(0,T) \right)|t-s|^{\alpha/3}
    \end{align*}
Then $\eta\le 1$ for $s,t$ satisfying the constraint $|t-s| \leq \tilde{L}$ where
    $$\tilde{L} = (C+1)\left( K^\alpha T^{2\alpha/3} + \norm{u}_{L^\infty_TH}^{\f 13}K^{\alpha/3}\omega_\mu^{\f 13}(0,T) \right)^{-\frac{3}{\alpha}} \wedge L.$$
  Conclusion follows from the optimality of the control $\omega_u$.
  \end{proof}

We now provide local estimates for the remainder $R^u$.
\begin{proposition}[remainder estimate]\label{prop:remainderestimate}
    Let $u$ be a URD solution to \eqref{eq:modeleqn} and $\tilde{L}$ the constant defined by (\ref{eq:Ltilde}). Suppose that $|s-t|\leq \tilde{L}$,   then
     \begin{equation} \label{est:remainder}
      \omega_R(s,t) := \norm{R^u}_{C^{\frac{1}{2\alpha}\var}([s,t];H^{-2})}^{\frac{1}{2\alpha}}
       \lesssim \left(1+ \norm{u}_{L_T^\infty H}^{\frac{1}{2\alpha}}\right) \left[ \omega_\mu^{\frac{1}{2\alpha}}(s,t) +\omega_A(s,t) \right].
    \end{equation}
\end{proposition}
\begin{proof}

    Let $\phi \in H^2$ and $0<\eta\le 1$, we estimate
    \begin{equation*}
        |\< R^u_{s,t}, \phi \>| \leq |\< R^u_{s,t}, J^\eta \phi \>| + |\< R^u_{s,t}, (Id - J^\eta)\phi \>|.
    \end{equation*}
For the first term, we use $R^u_{s,r} = \mu_{s,r} +A^{2}_{s,r}u_{s} + u_{s,r}^{\natural}$ with good temporal regularity, c.f. \eqref{usharpexpress}:
    \begin{align*}
        |\< R^u_{s,t}, J^\eta \phi \>| &\leq \omega_\mu(s,t) \norm{J^\eta \phi}_{H^1} + \norm{u}_{L_T^\infty H} \omega_A^{2\alpha}(s,t)\norm{J^\eta \phi}_{H^2}
        + \omega_\natural^{3\alpha}(s,t)\norm{J^\eta \phi}_{H^3} \\
        & \lesssim \left[ \omega_\mu(s,t)+  \norm{u}_{L_T^\infty H} \omega_A^{2\alpha}(s,t) + \eta^{-1}\omega_\natural^{3\alpha}(s,t)\right]\norm{\phi}_{H^2}.
    \end{align*}
    Similarly, using the better spatial regularity expression $R^u_{s,r} =  u_{s,r} - A^{1}_{s,r}u_{s}$ for the second term:
    \begin{align*}
        |\< R^u_{s,t}, (Id - J^\eta)\phi \>| &\leq \norm{u}_{L_T^\infty H}\norm{(Id - J^\eta)\phi}_H + \omega_A^\alpha(s,t)\norm{(Id - J^\eta)\phi}_{H^1}\\
        & \lesssim \left[ \norm{u}_{L_T^\infty H} \eta^2 + \omega_A^\alpha(s,t) \eta \right] \norm{\phi}_{H^2}.
        \end{align*}
 Combing the two yields:
    \begin{equation*}
        \norm{R^u_{s,t}}_{H^{-2}} \lesssim \omega_\mu(s,t)+  \norm{u}_{L_T^\infty H} \omega_A^{2\alpha}(s,t) + \eta^{-1}\omega_\natural^{3\alpha}(s,t) 
        + \norm{u}_{L_T^\infty H} \eta^2 + \omega_A^\alpha(s,t) \eta . 
    \end{equation*}
    Let us now take $\eta = \omega_A^\alpha(s,t)$ and assume that $|s-t|\le \f 1 K$. Recall that $L\le \f 1K$. Using this and  the definition $ \omega_R(s,t) := \norm{R^u}_{C^{\frac{1}{2\alpha}\var}([s,t];H^{-2})}^{\frac{1}{2\alpha}}$ we obtain:
    \begin{align*}
        \norm{R^u_{s,t}}_{H^{-2}} \lesssim (1+ \norm{u}_{L_T^\infty H})\left[\omega_\mu(s,t)+  \omega_A^{2\alpha}(s,t) \right].
    \end{align*}
   Taking supremum over partitions of $[s,t]$, the result follows.
\end{proof}

We now show that we can use Proposition~\ref{prop:solutionest} to obtain a global bound on the solution.
\begin{proposition}[global solution estimate] \label{prop:globalsolnest}
    Let $u$ be a URD solution to \eqref{eq:modeleqn} and $\tilde{L}$ the constant defined by (\ref{eq:Ltilde}).
    Then $u\in C^{\frac{1}{\alpha}\var}([0,T]; H^{-1})$:
    \begin{equation}\label{est:globalsolution}
       \norm{u}_{C^{\frac{1}{\alpha}\var}([0,T];H^{-1})}^{\frac{1}{\alpha}} \lesssim \left(\frac{T}{\tilde{L}}\right)^{\frac{1}{\alpha}-1}(1+\norm{u}_{L_T^\infty H})^{\frac{1}{\alpha}}\left[ 
            \omega_\mu^{\frac{1}{\alpha}}(0,T) + \omega_{\mu}^{\frac{1}{3\alpha}}(0,T)\omega_{A}^{\f 13}(0,T) \right].
    \end{equation}
\end{proposition}
\begin{proof}
    It suffices to obtain uniform  bound on $ \sum_{[s,t]\in \pi}\norm{u_{s,t}}_{H^{-1}}^{\frac{1}{\alpha}} $
    over all partitions $\pi$ of $[0,T]$. 
   As $\frac{1}{\alpha} \geq 1$, we can apply Jensen's inequality on each interval of the partition to refine the partition so that the mesh is reduced to below  $\tilde{L}$, at a cost of picking up a factor depending on the number of sub-intervals of length $\tilde L$ in $[0,T]$. Let $\pi '$ be a refinement of $\pi$. Then,
    \begin{equation*}
        \sum_{[s,t]\in \pi}\norm{u_{s,t}}_{H^{-1}}^{\frac{1}{\alpha}} \leq \left(\frac{T}{\tilde{L}}\right)^{\frac{1}{\alpha}-1}
        \sum_{[s,t]\in \pi'}\norm{u_{s,t}}_{H^{-1}}^{\frac{1}{\alpha}} \leq \left(\frac{T}{\tilde{L}}\right)^{\frac{1}{\alpha}-1} \sum_{[s,t]\in \pi'}\omega_u(s,t).
    \end{equation*}
   By the local solution estimate in Proposition~\ref{prop:solutionest} and the superadditivity of the controls, 
    \begin{align*}
        \sum_{[s,t]\in \pi}\norm{u_{s,t}}_{H^{-1}}^{\frac{1}{\alpha}} &\lesssim \left(\frac{T}{\tilde{L}}\right)^{\frac{1}{\alpha}-1} (1+\norm{u}_{L_T^\infty H})^{\frac{1}{\alpha}}
        \sum_{[s,t]\in \pi'}\left[ \omega_\mu^{\frac{1}{\alpha}}(s,t) + \omega_{\mu}^{\frac{1}{3\alpha}}(s,t)\omega_{A}^{\f 13}(s,t) \right]\\
        &\lesssim  \left(\frac{T}{\tilde{L}}\right)^{\frac{1}{\alpha}-1}(1+\norm{u}_{L_T^\infty H})^{\frac{1}{\alpha}}\left[ 
            \omega_\mu^{\frac{1}{\alpha}}(0,T) + \omega_{\mu}^{\frac{1}{3\alpha}}(0,T)\omega_{A}^{\f 13}(0,T) \right].
    \end{align*}
 Note that the right-hand side is uniform in $\pi$, so taking supremum over partitions concludes the proof.
\end{proof}

Finally, we want to upgrade to a global in time estimate on $R^u\in C^{\frac{1}{2\alpha}\var}([0,T]; H^{-2})$. 
\begin{proposition}[global remainder estimate] \label{prop:globalremainderest}
    Let $u$ be a URD solution to \eqref{eq:modeleqn}. Then $R^u\in C^{\frac{1}{2\alpha}\var}([0,T]; H^{-2})$. Moreover:
    \begin{equation*}
        \norm{R^u}_{C^{\frac{1}{2\alpha}\var}([0,T];H^{-2})}^{\frac{1}{2\alpha}}
         \lesssim  
        \left(1+ \norm{u}_{L_T^\infty H}^{\frac{1}{2\alpha}}\right) \left[ \omega_\mu^{\frac{1}{2\alpha}}(0,T) +\omega_A(0,T) \right]+
        \frac{T^{3/2}}{\tilde{L}} \omega_u^{\f 12}(0,T).
    \end{equation*}
\end{proposition} 
\begin{proof}
    Let $\pi$ be a partition of $[0,T]$. Refine $\pi$ until the mesh is less than $\tilde{L}$ so that the local estimate applies.
    To this end, we have to understand what happens when we add a point to the partition. Note the following:
    \begin{equation*}
        \norm{R^u_{s,t}}_{H^{-2}}^{\frac{1}{2\alpha}} = \norm{R^u_{s,r} + R^u_{r,t} - \delta R^u_{s,r,t}}_{H^{-2}}^{\frac{1}{2\alpha}} \lesssim
        \norm{R^u_{s,r}}_{H^{-2}}^{\frac{1}{2\alpha}} + \norm{R^u_{r,t}}_{H^{-2}}^{\frac{1}{2\alpha}} + \norm{\delta R^u_{s,r,t}}_{H^{-2}}^{\frac{1}{2\alpha}}.
    \end{equation*}
   By direct calculation $\delta R^u_{s,r,t} = (Z_{r,t} \cdot \nabla)u_{s,r}$, so: 
    \begin{equation*}
        \norm{\delta R^u_{s,r,t}}_{H^{-2}}^{\frac{1}{2\alpha}} \lesssim \norm{Z_{r,t}}_{H^\sigma}^{\frac{1}{2\alpha}} \norm{u_{s,r}}_{H^{-1}}^{\frac{1}{2\alpha}}
        \lesssim |t-r|^{\f 12}\omega^{\f 12}_u(s,r) \lesssim T^{\f 12} \omega_u^{\f 12}(0,T).
    \end{equation*}
    Since we only need to add at most $\frac{T}{\tilde{L}}$ points to the partition, we obtain:
    \begin{equation*}
        \sum_{[s,t]\in\pi} \norm{R^u_{s,t}}_{H^{-2}}^{\frac{1}{2\alpha}} \lesssim \sum_{[s,t]\in\pi'} \norm{R^u_{s,t}}_{H^{-2}}^{\frac{1}{2\alpha}} +
        \frac{T}{\tilde{L}}T^{\f 12} \omega_u^{\f 12}(0,T).
    \end{equation*}
   Apply the local remainder estimate, Proposition~\ref{prop:remainderestimate}, and use superadditivity of controls to see
    \begin{equation*}
        \sum_{[s,t]\in\pi} \norm{R^u_{s,t}}_{H^{-2}}^{\frac{1}{2\alpha}} \lesssim  
        \left(1+ \norm{u}_{L_T^\infty H}^{\frac{1}{2\alpha}}\right) \left[ \omega_\mu^{\frac{1}{2\alpha}}(0,T) +\omega_A(0,T) \right]+
        \frac{T}{\tilde{L}}T^{\f 12} \omega_u^{\f 12}(0,T).
    \end{equation*}
    As before, the right-hand side is uniform in the partition $\pi$, hence we can take a supremum over the set of all partitions concluding
    the proof. 
\end{proof}

\subsection{Equivalence of solutions}

We can now prove Theorem~\ref{thm:equivalence-of-sol}, demonstrating the equivalence of notions of solutions defined by Definition~\ref{def:RPsolution} and \ref{def:URDsoln}.
\begin{proposition}\label{pro:Rough path solutions are URD solutions}
    Assume $u$ is a rough path solution. Then it is a URD solution.
\end{proposition}
\begin{proof}    
    By Proposition~\ref{rem:distrvstested}, the rough integral
    $$-\int_s^t \< \Pi[(d{\bf Z}_r \cdot \nabla )u_r] ,\phi \> =\lim_{|\pi|\to 0} \sum_{[s', t']\in \pi}   y_{s'}Z_{s',t'} + y'_{s'}\mathbb{Z}_{s',t'}$$
 is well-defined via sewing. Consequently, the residual:
    \begin{equs} I^\natural_{s,t}(\phi) :&= \< u_{s,t}, \phi \> - \int_s^t \<\nu \Delta u_r, \phi  \> dr + \int_s^t \< \Pi[(u_r \cdot \nabla) u_r], \phi \> dr 
    -  \left[ y_s Z_{s,t} + y'_s\mathbb{Z}_{s,t} \right] \\
    &=\int_s^t \< \Pi[(d{\bf Z}_r \cdot \nabla )u_r] ,\phi \>-   \left[ y_s Z_{s,t} + y'_s\mathbb{Z}_{s,t} \right]
    \end{equs}
belongs to finite $\frac{1}{3\alpha}$-variation space on $[0,T]$. Furthermore,  by the sewing lemma and \eqref{eq:increment-est}
$$|I^\natural_{s,t}(\phi)|\le C   \left| \delta H_{s,u,t} \right|  \le C \omega(s,t)^{3\alpha}\norm{\phi}_{H^3},$$
where $\omega(s,t) $ is as in the proof of Proposition~\ref{rem:distrvstested}.
 This allows us to identify $I^\natural_{s,t}$ with a distribution in $H^{-3}$, with finite $C^{\frac{1}{3\alpha}\var}([0,T]; H^{-3})$ norm (see Proposition~\ref{rem:distrvstested}). Unraveling notations shows that this shows that $u$ is a URD solution.
\end{proof}

\begin{proposition}\label{pro:URD solutions are rough path solutions}
    Assume that $u$ is a URD solution. Then it is a rough path solution.
\end{proposition}
\begin{proof}
Let $u$ denote a URD solution. By  Proposition~\ref{prop:globalsolnest} it has the required additional H\"older regularity:  $u\in C^{\frac{1}{\alpha}\var}([0,T]; H^{-1})$. Proposition~\ref{prop:globalremainderest} established the required regularity requirement for $R^u_{s,t}$,
it belongs to $C^{\frac{1}{2\alpha}\var}_{2}([0,T];H^{-2})$.
It remains to `upgrade' the weak formulation from \eqref{eq:URDweak_formulation} to \eqref{eq:weak_formulation}. 
    To this end let us define the increments:
    \begin{equation*}
        H_{s,t} :=  - \left[ y_s Z_{s,t} + y'_s\mathbb{Z}_{s,t} \right],
    \end{equation*}
    and\begin{equation*}
        \hat{A}_{s,t} := \< u_{s,t}, \phi \>  - \int_s^t \<\nu \Delta u_r, \phi  \> dr + \int_s^t \< \Pi[(u_r \cdot \nabla) u_r], \phi \> dr
    \end{equation*}
    using notations of Proposition~\ref{rem:distrvstested}.  Let us now note, that by definition
    \begin{equation*}
        \< u^\natural_{s,t}, \phi \> = \hat{A}_{s,t} + H_{s,t}.
    \end{equation*}
      Since $\hat{A}$ is a $2$-parameter process arising as a difference of a continuous function, and $H$ satisfies the conditions of the sewing lemma (as in the proof of Proposition~\ref{rem:distrvstested}), we have that there exists a unique process $I$ and $I^\natural$ of finite $\frac{1}{3\alpha}$-variation such that:
    \begin{equation} \label{sewing-for-residual}
        \delta I_{s,t} = \hat{A}_{s,t} + H_{s,t} + I^\natural_{s,t}= \< u^\natural_{s,t}, \phi \> + I^\natural_{s,t}.
    \end{equation}
    In particular, we note that 
    \begin{equation} \label{RP-int-from-URD}
     H_{s,t} + I^\natural_{s,t} = -\int_s^t \< \Pi[(d{\bf Z}_r \cdot \nabla )u_r] ,\phi \>, 
     \end{equation}
    and by the second inequality in \eqref{sewing-for-residual} we have that $I$ is a continuous function of finite $\frac{1}{3\alpha}$-variation on the intervals of the partition from the definition of URD solution. This together with $I_0=0$ implies that $I$ is identically zero on each sub-interval, so also on $[0,T]$. This shows, that $$ \< u^\natural_{s,t}, \phi \> = - I^\natural_{s,t} $$ on $[0,T]$, which implies that $u^{\natural}\in C^{\frac{1}{3\alpha}\var}([0,T]; H^{-3})$, as linearity follows from the linearity of the increments, and boundedness is a consequence of the sewing bound, as in Proposition~\ref{rem:distrvstested}. The weak formulation \eqref{eq:weak_formulation} now follows from
    \eqref{RP-int-from-URD} and the first equality in \eqref{sewing-for-residual}.
   \end{proof}

\section{Rough functional limit theorem for the fast process} \label{sec:thmA}

Let $w^\epsilon$ denote the stationary solution of the following equation on $H^\sigma$
\begin{equation}
     dw^\epsilon = -\epsilon^{-1} M w^\epsilon dt + 
    \epsilon^{-H} Q^{\f 12} d\hat B^H.
\end{equation}
 
\begin{assumption}\label{assumption}
    \begin{enumerate}
        \item [(i)] The operators
        $M, Q \in \mathcal{L}( H^{\sigma}; H^{\sigma})$ are non-negative, symmetric and commute. 
        \item [(ii)]  $M$  generates a strongly continuous semigroup and
        $\norm{M}_{\L(H^{\sigma})} + \norm{M^{-1}}_{\L(H^{\sigma})} \lesssim 1$.
        \item [(iii)] $Q$ is  trace class operator, and  $Q^{\beta}$ is trace class where $\beta=2H$ in case 
        $H\in (\f 13, \f12)$ and $\beta=\f 	12$ in case $H\in (\f 12, 1)$.
    \end{enumerate}
\end{assumption}
Throughout Section~\ref{sec:thmA} we work under Assumption~\ref{assumption}. The aim of the section is to prove Theorem~\ref{thm:A}, recalled below.
\begin{theorem}
Let $M$ and $Q$ satisfy Assumption \ref{assumption}, $H\in \left(\frac{1}{3}, \frac{1}{2}\right)$, and
    \begin{equation*}
        X_{s,t}^\epsilon=\epsilon^{H-1} \int_s^t  w_r^\epsilon dr, \quad \mathbb{X}_{s,t}^\epsilon = 
        \int_s^t X_{s,r}^\epsilon \otimes dX_{r}^\epsilon
    \end{equation*}
    Then $(X^{\epsilon}_{\cdot}, \mathbb{X}^{\epsilon}_{ \cdot\cdot})$ converges in $L^{q}(\Omega;\mathcal{C}^{\alpha}
    ([0,T];H^{\sigma}))$ to $(B,\mathbb{B})$ for all $1<q<\infty$ and $\frac{1}{3}<\alpha< H$.
\end{theorem}

The convergence of the second order process is more involved for which we bound the finite variation norm of the $2$-d covariance of the relevant processes, and use the machinery introduced in Section~\ref{sec:gaussianrp}.

\subsection{Spectral decomposition and the stationary fOU}
We fix the spectral notation and collect the basic $L^2$
  and covariance bounds for the stationary fractional Ornstein–Uhlenbeck process that will be used throughout the proof.
   
Let $\{e_i\}_{i=1}^\infty$ denote an orthonormal basis of $H^\sigma$ consisting of  eigenvectors of $M$ and $Q$ with corresponding 
eigenvalues $c_i$ and $\lambda_i$ respectively,  with $c_i$ arranged in increasing order.
 With this we represent the $Q$-fractional Brownian motion as 
\begin{equation*}
    Q^{\frac{1}{2}} B_{t}^{H}= \sum_{i \geq 1}\sqrt{\lambda_i}e_i B_t^{H,i}
\end{equation*}
where $\{ B_t^{H,i}\}$  be independent real valued fractional Brownian motions with Hurst parameter~$H$.  In particular, 
$B$ admits a canonical Gaussian lift (Section~\ref{sec:gaussianrp}).
Further, we denote:
\begin{equation}
    B_t:= M^{-1}Q^{\frac{1}{2}} B^H_t = \sum_i e_i B^{i}_t
\end{equation}
 where 
$$B^i :=  \f{ \sqrt{\lambda_i}}{c_i} B^{H,i}_t.$$
The process $B$ will be the limiting rough path appearing in Theorem~\ref{thm:A}.

We begin with a bound on the non-rescaled stationary fOU process: $dw_t = -Mw_t dt + Q^{\f 12}dB^H_t$.
The convergence of the first level process, Proposition \ref{prop:first-order}, and the uniform in time bound on the non-rescaled 
stationary fOU process given in Lemma \ref{lem:uniformboundFOU} generalises that in \cite[Prop. 3.1]{Gehringer-Li-fOU}, 
\cite{gehringer-li2}. Since $Q$  is trace class, the arguments extend verbatim; one only needs to keep track of the dependence on the eigenvalues $\lambda_i$.

Let $$C_H= \frac{\Gamma(2H+1)\sin(\pi H)}{2\pi}\int_\R \frac{|x|^{1-2H}}{1 + x^2} dx.$$

\begin{lemma} \label{lem:uniformboundFOU}
    Let $w$ be the stationary solution of $dw_t = -Mw_t dt + Q^{\f 12}dB^H_t$. Then:
\begin{equ}\label{uniform-bound}
      \sup_{t>0} \mathbb{E}\left[\left\|w_t\right\|_{H^\sigma}^2\right] = C_H \tr{M^{-2H}Q}, \qquad 
       \sup_{t>0} \mathbb{E}\left[\norm{M^{-1}w_{t}}_{H^{\sigma}}^{2}\right] = C_H\tr{M^{-2-2H}Q}.
\end{equ}
Furthermore $w$ satisfies the bounds: 
$$\mathbb{E}\left[\left\|w_{s,t}\right\|_{H^\sigma}^2\right]  \lesssim  \left(\tr{M^{2-2H}Q} +\tr{Q} + \tr{M^{-2H}Q} \right) |t-s|^{2H},$$
$$\mathbb{E}\left[\left\|M^{-1} w_{s,t}\right\|_{H^\sigma}^2\right]  \lesssim  
\left(\tr{M^{-2H}Q} +\tr{M^{-2}Q} + \tr{M^{-2-2H}Q}\right) |t-s|^{2H},$$
where the constants indicated by $\lesssim$ depend only on $H$.
\end{lemma}

\begin{proof} The equation for $w$ on $H^\sigma$ is equivalent to the following system of equations  
$$ dw^{i} = -c_i w^{i}dt + \sqrt{\lambda_i}d B^{i,H}_t, \quad i=1,2,\dots,$$
where we write  $w^{i} := \< w, e_i\>$. For each $i$, the equation has a stationary 
solution $\sqrt{ \lambda_i} \int_{-\infty}^t e^{-c_i(t-s)} d B_s^{H, i}$
whose variance is $\f{\lambda_i}{c_i^{2H}}\frac{\Gamma(2H+1)\sin(\pi H)}{2\pi} 
\int_\R \frac{|x|^{1-2H}}{1 + x^2} dx$. Therefore  
  \begin{equ}
         \mathbb{E}\left[\left|w^{i}_t\right|^{2}\right] 
       =C_H \f{\lambda_i}{(c_i)^{2H}},     \qquad  \mathbb{E}\left[\norm{w_{t}}_{H^{\sigma}}^{2}\right] 
          = \sum_{i =1} ^\infty  \mathbb{E}\left[\left|w^{i}_t\right|^{2}\right] =
          C_H \sum_{i=1}^\infty \f{\lambda_i}{(c_i)^{2H}}=C_H \tr{M^{-2H}Q}.
    \end{equ}
     Using the SDE, Jensen's inequality, and Fubini Theorem to exchange the integrations we get:
      \begin{align*}
        \mathbb{E}\left[\left|w^{i}_{s,t}\right|^2\right] &\le 2 \mathbb{E}\left[\left|\int_{s}^{t}-c_iw^i_rdr\right|^2\right]
        + 2{\lambda_i}\mathbb{E}\left[\left|B^{H,i}_{s,t}\right|^2\right] 
        \lesssim 2 c_i^2 |t-s|^2\sup_{t\geq 0}\mathbb{E}[|w^i_t|^2]
        +2 {\lambda_i} |t-s|^{2H},  \end{align*}
where the implicit constants depend only on $H$. Consequently         
\begin{equs}
\mathbb{E}\left[\norm{w_{s,t}}_{H^\sigma}^2\right]  &\lesssim \sum_i \lambda_i c_i^{2-2H} C_H |t-s|^2 +\sum_i{\lambda_i}|t-s|^{2H}\\
&\lesssim \tr{M^{2-2H}Q} |t-s|^2 +\tr{Q}|t-s|^{2H},
\end{equs} 
which gives the required bound for $|t-s|\leq 1$.
Combining with the uniform bound $\E(\|w_{s,t}\|^2) \lesssim \tr{M^{-2H}Q}$ when $|t-s|\ge 1$, 
we obtain the required estimates. The estimates for $M^{-1}w_{s,t}$ are obtained similarly.
\end{proof}

\subsection{First-level convergence}

\begin{proposition} \label{prop:first-order} Let  $H\in (\f 13,1) \backslash \{ \frac{1}{2} \}$. Then $X^{\epsilon}_{\cdot}$ converges to $B_{\cdot}$ in
     $L^{q}(\Omega;C^{\alpha}([0,T];H^{\sigma}))$ for any $q>1$ and $\alpha< H -\frac{1}{q}$. 
\end{proposition}

\begin{proof}
From the definition
$$
X_{s,t}^\epsilon=\epsilon^{H-1}\int_s^t w_r^\epsilon\,dr,
$$
and using the equation satisfied by $w^\epsilon$, a direct computation yields the identity
\begin{equation}\label{eq:key-identity}
    X_{s,t}^\epsilon
    =
    M^{-1}Q^{\frac12}B^H_{s,t}
    - \epsilon^H M^{-1}\big(w_t^\epsilon-w_s^\epsilon\big).
\end{equation}
Hence it suffices to estimate the second term.

Since $B^H$ is Gaussian, so are $X^\epsilon_{s,t}$ and 
$M^{-1}Q^{\frac12}B^H_{s,t}$. By Gaussian hypercontractivity, it is enough to work in $L^2$.

Using Lemma~\ref{lem:uniformboundFOU} together with the stationarity property
$w_t^\epsilon \eqL w_{t/\epsilon}$, we obtain the two bounds
$$
    \mathbb{E}\|M^{-1}(w_t^\epsilon-w_s^\epsilon)\|_{H^\sigma}^2
    \lesssim
    \begin{cases}
        1, \\
       \epsilon^{-2H} |t-s|^{2H},
    \end{cases}
$$
where the implicit constants depend only on
$\tr{M^{-2-2H}Q}, \tr{M^{-2H}Q}$ and $\tr{M^{-2}Q}$.

Multiplying by $\epsilon^{2H}$ and interpolating the two estimates yields, for any $\gamma<H$,
$$
    \|X^{\epsilon}_{s,t}-M^{-1}Q^{\frac{1}{2}}B^{H}_{s,t}\|_{L^{2}(\Omega ; H^{\sigma})}
    \lesssim
    \epsilon^{H-\gamma}|t-s|^{\gamma}.
$$

By equivalence of Gaussian moments and the Kolmogorov criterion, for any
$\alpha<\gamma-\frac{1}{q}$,
$$
    \Big\|
    X^{\epsilon}_{\cdot}-M^{-1}Q^{\frac{1}{2}}B^{H}_{\cdot}
    \Big\|_{L^{q}(\Omega;C^{\alpha}([0,T];H^{\sigma}))}
    \lesssim
    \epsilon^{H-\gamma},
$$
which concludes the proof.
\end{proof}

\begin{remark}
The functional central limit theorem for the one-dimensional fractional Ornstein–Uhlenbeck process is classical \cite{Pipiras-Taqqu00,Boufoussi-Tudor05}. This was upgraded in \cite{Gehringer-Li-fOU} to convergence in the rough path topology, allowing the use of tools from rough path analysis. In the one-dimensional setting, the rough path topology coincides with Hölder convergence. See also \cite{volterra} for equations of the form 
\(\dot x_t=\epsilon^{-\frac12}f(x_t^\epsilon, w_t^\epsilon)\),
where \(w_t^\epsilon\) is a generic finite dimensional Volterra process.

In \cite{gehringer-Li-19,gehringer-li2}, multiple interaction terms require convergence of the corresponding Lévy areas. In particular, the rescaled processes 
$
\epsilon^{-\frac12}\int_s^t G_i(w_r^\epsilon)\,dr
$
and their canonical rough path lifts converge to a Wiener process with non-trivial covariance provided the Hermite rank \(m_i\) of \(G_i\) satisfies
$
(H-1)m_i+1<\tfrac12$.

In the present setting \(G_i=\mathrm{Id}\), hence \(m_i=1\). The above condition is satisfied for all \(H<\tfrac12\); however, under the diffusive scaling \(\epsilon^{-1/2}\) the limiting covariance vanishes and the limit process is trivial. The correct non-degenerate scaling is instead \(\epsilon^{H-1}\), see also \cite{Gehringer-Li-fOU} for the one-dimensional case. \end{remark}

\subsection{Proof of the rough functional limit theorem}\label{subsec:pf-thmA}

For convenience of the reader we briefly recall the key objects: the norms of the $2$-d covariance introduced in \eqref{covariance-norm}:
$$
\norm{R_X}_{\rho; I \times I'}, \qquad 
\norm{R_X}_{\rho\text{-var};I \times I'}.$$
We will use Theorem~\ref{thm:FH10.5}, which shows that rough path distance is controlled by the finite variation norms of the$2$-d covariance, together with Theorem~\ref{thm:CQ} to estimate these norms for the processes considered below.

\begin{proof}[Proof of Theorem~\ref{thm:A}]
We prove the convergence of the  processes:
$$ X_{s,t}^\epsilon=\epsilon^{H-1} \int_s^t  w_r^\epsilon dr, \quad \mathbb{X}_{s,t}^{\epsilon,i,j} = \int_s^t X^{\epsilon, i}_{s,r} \otimes 
      dX^{\epsilon, j}_{r},$$
      where  $w^\epsilon$ is the stationary fOU solving $dw^\epsilon = -\epsilon^{-1} M w^\epsilon dt + 
    \epsilon^{-H} Q^{\f 12} dB^H$.

   The idea is to apply Theorem \ref{thm:FH10.5} with $X^\epsilon$ and $B$ as the two Gaussian processes.
     For this, we control the finite variation norms of the $2$-d covariance of $X^\epsilon$, $B$ and $X^\epsilon - B$, see Definition~\ref{covariance-norm}. 
     The norms on $2$-d covariance of a Gaussian process are not easy objects to control directly, however, we can employ
 Theorem~\ref{thm:CQ} stating that these norms can be controlled by the relevant variance and covariance of the increments  \eqref{Coutin-Qian} of the processes, which we calculate below.
 
 Let us introduce the non-rescaled process $X$ defined as
 $X_{s,t} = \int_s^t w_r dr, $
 since then 
 $$X^\epsilon_{s,t}= \epsilon^{H-1}\int_{s}^{t}w^\epsilon_rdr \stackrel{\mathrm{law}}= \epsilon^H
\int_{\frac{s}{\epsilon}}^{\frac{t}{\epsilon}} w_r dr .$$
   
\noindent\textbf{Step 1.}
Verify covariance bounds for $B$ and $X^\epsilon$ splitting into the proof for the variance and for the increment covariance. 
    Since 
$\norm{R_{B^i}}_{\frac{1}{2H};[s,t]^2} \lesssim (\lambda_i + \lambda_i^{2H})|t-s|^{2H}$,  by  \cite[Cor. 10.10]{FH} and the assumption $Tr[Q^{2H}] < \infty$,
 $B$ has a canonical Gaussian lift.

 To control the $2$-d covariance of $X^\epsilon$, we verify the conditions in Theorem~\ref{thm:CQ}.
 Let $w^i$ denote the projection of $w$ onto $e_i$ obtaining 1-d real valued fOU processes.
  By Lemma \ref{lem:uniformboundFOU}, dropping $M^{-1}$ in the bound, we have
\begin{equation*}
    \mathbb{E}\left[\left|w^i_{s,t}\right|^2\right] \lesssim \lambda_i |t-s|^{2H}.
\end{equation*}
Then from the fOU SDE, 
$w_{s,t}^\epsilon = -\epsilon^{-1} M \int_s^t w_r ^\epsilon \,dr+ 
    \epsilon^{-H} Q^{\f 12} B_{s,t}^H$,
    $$X^\epsilon_{s,t} =B_{s,t} - \epsilon^{H} M^{-1} w^\epsilon_{s,t}$$
By scaling of $w$, and the equivalence in law of $(w^\epsilon_t)_{t\ge0}\stackrel{\mathrm{law}}=(w_{t/\epsilon})_{t\ge0}$, we obtain
\begin{equation*}
    \mathbb{E}\left[\left|X^{\epsilon, i}_{s,t}\right|^2\right] \lesssim \lambda_i|t-s|^{2H}+\lambda_i \epsilon^{2H} \left|\frac{t-s}{\epsilon}\right|^{2H} \lesssim 
 \lambda_i |t-s|^{2H}.
\end{equation*}

 To obtain estimates on the covariance of the increments, we use the correlation decay of the fOU process.
 Let  $\nu^{i}(s)=\mathbb{E}[w^{i}_{s}w^{i}_{0}]$ denote the covariance of~$w^i$. Now, 
\begin{equation*}
    \nu^{i}(s)=\lambda_{i}\frac{\Gamma(2H+1)\sin(\pi H)}{2\pi}\int_{\mathbb{R}}e^{isx}\frac{|x|^{1-2H}}{|c_{i}|^{2}+x^{2}}dx
\end{equation*}
has  an asymptotic expansion as $s\rightarrow \infty$ (see \cite{fou_cheridito}):
\begin{equation*}
    \nu^{i}(s)=\frac{1}{2}\lambda_{i} \sum_{n=1}^{N}c_{i}^{-2n} ( \prod_{k=0}^{2n-1} (2H-k))s^{2H-2n} + O(s^{2H-2N-2}).
\end{equation*}
In particular, we have $|\nu^i(s)|\lesssim \frac{\lambda_i}{c_i^2} 1 \wedge |s|^{2H-2} \lesssim \lambda_i (1 \wedge |s|^{2H-2})$.
Since $M$ is bounded, we have dropped the $c_i$'s.

For the second condition in \eqref{Coutin-Qian} for the $X$ process, we calculate, for $0<h<t-s$:
\begin{align*}
    \left|\mathbb{E} \left[X^i_{s,s+h}X^i_{t,t+h}\right]\right| & = \left|\int_{s}^{s+h} \int_t^{t+h} \nu^i(r-r')drdr'\right|
    \lesssim \int_{s}^{s+h} \int_t^{t+h} \lambda_i  |r-r'|^{2H-2}\,drdr'\\
    &\lesssim \lambda_i \int_0^h \int_0^h |r-r'+t-s|^{2H-2}drdr' \\
    &\lesssim \lambda_i |t-s|^{2H-2} h^2,
\end{align*} 
using the covariance asymptotics for the first inequality, and changing variables. Again, by rescaling we obtain
\begin{equation*}
    \left|\mathbb{E} \left[X^{\epsilon,i}_{s,s+h}X^{\epsilon, i}_{t,t+h}\right]\right| \lesssim \lambda_i \epsilon^{2H}
    \left|\frac{t-s}{\epsilon}\right|^{2H-2}\left(\frac{h}{\epsilon}\right)^2 = \lambda_i |t-s|^{2H-2} h^2.
\end{equation*}
An application of Theorem~\ref{thm:CQ} now yields
$$ \norm{R_{X^{\epsilon, i}}}_{\frac{1}{2H}\var;[s,t]^2} \lesssim \lambda_i |t-s|^{2H}. $$

\noindent\textbf{Step 2}
To verify bounds \eqref{Coutin-Qian} for $X^\epsilon - B$ with smallness in $\epsilon$, it is crucial to use the relation given by the fOU SDE 
$$X^\epsilon_{s,t} - B_{s,t} = \epsilon^{H} M^{-1} w^\epsilon_{s,t}$$

Because $H<\f 1 2$, the kernel $|r-r'|^{2H-2}$ is singular and direct computation of the covariance with overlapping integration areas is difficult to implement. Fortunately, we have at our disposal the recent result  from  \cite[Corollary 3.4]{Qian-Xu-2024}: the $2$-d covariance of the fOU process $w^i$ has finite $\frac{1}{2H}$-variation norm and  
$$ \norm{R_{w^i}}_{\frac{1}{2H}\var; [s,t]^2} \lesssim \lambda_i |t-s|^{2H}. $$ 
Rescaling time, yields a uniform in $\epsilon$ bound for $\epsilon^{H} M^{-1} w^\epsilon_{s,t}$, equivalently for $X^\epsilon - B$ :
$$  \norm{R_{X^{\epsilon, i}-B^i}}_{\frac{1}{2H}\var} \lesssim \lambda_i |t-s|^{2H}. $$
Furthermore, by Lemma \ref{lem:uniformboundFOU} and stationarity, we have a uniform in time bound of order~$\epsilon^{2H}$:
\begin{equation*}
    \norm{R_{X^{\epsilon, i}-B^i}}_{\infty;[s,t]^2} \lesssim \epsilon^{2H}
    \sup_{u,v,u',v'\in[s,t]}\left|\mathbb{E}\left[w^{\epsilon,i}_{u,v}w^{\epsilon, i}_{u',v'}\right]\right| \lesssim 
    \epsilon^{2H} \sup_{t\geq 0} \mathbb{E}[|w^i_t|^2] \lesssim \epsilon^{2H} \lambda_i,
\end{equation*}
which follows by triangle inequality and H\"older inequality.

\noindent\textbf{Step 3} 
 We can now interpolate the variation norm with the supremum norm to obtain the $\rho'$ -variation bound with smallness in $\epsilon$. For $\rho'>\frac{1}{2H}$
\begin{align*}
    \norm{R_{X^{\epsilon, i}-B^i}}_{\rho'-var;[s,t]^2} &\leq \norm{R_{X^{\epsilon, i}-B^i}}_{\frac{1}{2H}-var;[s,t]^2}^{\frac{1}{2H\rho'}}
    \norm{R_{X^{\epsilon, i}-B^i}}_{\infty;[s,t]^2}^{1-\frac{1}{2H\rho'}} \\
    &\lesssim \lambda_i |t-s|^{\frac{1}{\rho'}} \epsilon^{2H-\frac{1}{\rho'}}.
\end{align*} 

\noindent\textbf{Step 4.}
Finally, we choose $\rho'>\frac{1}{2H}$ so that $2H-\frac{1}{\rho'}>0$.  
Apply Theorem~\ref{thm:FH10.5} with $X=X^\epsilon$, $Y=B$, $\rho=\rho'$ and
$M_i:=\lambda_i+\lambda_i^{2H}\in\ell^1$ (by the trace assumptions). 
Step~3 yields the required smallness with $\epsilon^{2H-\frac{1}{\rho'}}$ in place of $\epsilon^2$, which concludes the proof.
\end{proof}

\section{Uniform bounds and proof of the main theorem} \label{sec:thmB}
Once the convergence in the rough topology of the  processes
\begin{equ}\label{X4}
 X_{s,t}^\epsilon=\epsilon^{H-1} \int_s^t  w_r^\epsilon dr, \quad \mathbb{X}_{s,t}^{\epsilon,i,j} = \int_s^t X^{\epsilon, i}_{s,r} \otimes 
      dX^{\epsilon, j}_{r},
      \end{equ}
  is established, we proceed to Theorem~\ref{thm:B}. We shall in fact prove a more general result with $X^\epsilon$ replaced by a more general fast component $Z^\epsilon$. To this end, we consider the modified equation: \begin{equ}\label{sf2}
  \partial_t u^\epsilon = \nu \Delta u^\epsilon - \Pi[(u^\epsilon \cdot \nabla) u^\epsilon] -   \Pi[(d\mathbf Z^\epsilon \cdot \nabla) u^\epsilon],
\end{equ} 
where $\mathbf Z^\epsilon$  satisfies the assumption below.
\begin{assumption}
\label{ass:Y}
 Assume that $(Z^{\epsilon}_{\cdot}, \mathbb{Z}^{\epsilon}_{ \cdot\cdot})$ is a family of $\alpha$-H\"older geometric rough paths where $\alpha\in (\f 13, 1)$
 such that as $\epsilon \to 0$,
  $\mathbf Z^\epsilon \equiv (Z^{\epsilon}, \mathbb{Z}^{\epsilon})\to  \mathbf Z $ in $L^{q}(\Omega;\mathcal{C}^{\alpha}([0,T];H^{\sigma})
\times \mathcal{C}^{2\alpha}(\Delta_T;H^{\sigma}\otimes H^{\sigma}))$.
\end{assumption}

Throughout the rest of the section,  $u_t^\epsilon$ denotes the solution of \eqref{sf} or \eqref{sf2} with initial data satisfying $\sup_{\epsilon>0}\norm{u^\epsilon_0}_{H} <\infty$.
We let $X^\epsilon$ denote either $Z^\epsilon$ or the expression in \eqref{X4}. 
 \medskip
 
 In case of \eqref{X4}, we may assume that $H\in \left(\frac{1}{3}, \frac{1}{2}\right)$. The case $H\in (\f 12, 1)$ is easier as it involves only Young integrals. Let us fix a set of full measure of the probability space on which 
the fBM is $\alpha <H$ H\"older continuous. From now on everything is understood pathwise on this set.

We first present the uniform estimates.
 \begin{lemma}
 The following energy estimate holds for the solutions, uniformly in $\epsilon$:
    \begin{equation}
        \label{energyestimate}
                \sup_{t\leq T}\norm{u^{\epsilon}_{t}}_{H}^{2} + \int_{0}^{T} \norm{\nabla u^{\epsilon}_{r}}_{H}^{2}dr 
            \leq C(\nu) \norm{u^{\epsilon}_{0}}_{H}^{2}.
    \end{equation}
    \end{lemma}
    \begin{proof} This is a standard energy estimate for the Navier-Stokes equations, facilitated by the purely transport nature of the noise.
     Denote by $\Pi_m$ the projection onto the first $m$ eigenfunctions of the Stokes operator, and consider the Galerkin approximation
        $$ \partial_{t}u^{\epsilon,m} = \nu \Delta u^{\epsilon,m} - \Pi_m \Pi[(u^{\epsilon,m} \cdot \nabla) u^{\epsilon,m}] - \epsilon^{H-1} \Pi_m \Pi[(w^{\epsilon} \cdot \nabla) u^{\epsilon,m}], \qquad u_0^{\epsilon,m} = \Pi_m u_0^\epsilon.$$
    Taking the $H$-inner product with $u^{\epsilon,m}$ and integrating by parts, we get:
    \begin{equation*}
        \frac{1}{2} \frac{d}{dt} \norm{u^{\epsilon,m}}_{H}^{2} + \nu \norm{\nabla u^{\epsilon,m}}_{H}^{2} = 0.
    \end{equation*}
The two advection terms vanish due to skew-symmetry. Integrating in time we obtain:
    \begin{equation*}
        \norm{u^{\epsilon,m}_{t}}_{H}^{2} + 2\nu \int_{0}^{t} \norm{\nabla u^{\epsilon,m}_{r}}_{H}^{2}dr = \norm{u^{\epsilon,m}_{0}}_{H}^{2} \leq \norm{u^{\epsilon}_{0}}_{H}^{2},
    \end{equation*}
    for any $t \leq T$. By weak lower semi-continuouity of the norm, for any accumulation point $u^{\epsilon}$ we have
    $$ \norm{u^{\epsilon}_{t}}_{H}^{2} + 2\nu \int_{0}^{t} \norm{\nabla u^{\epsilon}_{r}}_{H}^{2}dr\le \lim\inf_{m\to \infty} (\norm{u^{\epsilon,m}_{t}}_{H}^{2} + 2\nu \int_{0}^{t} \norm{\nabla u^{\epsilon,m}_{r}}_{H}^{2}dr),$$
yielding the conclusion.
    \end{proof}

\begin{remark}
We remark that the  family of solutions $\{ u^\epsilon\}$ indeed exists, which is covered by the deterministic theory pathwise, as they are smooth transport perturbations of the Navier-Stokes equation. Indeed, by the last lemma there exists a weakly convergent sequence $u^{\epsilon, m_j}$, 
in the energy space, and an application of a standard Aubin-Lions compactness lemma allows us to pass to the limit in every term of the weak formulation:
$$
        \< u^{\epsilon,m_j}_{s,t}, \phi \> = \int_s^t \<\nu \Delta u^{\epsilon,m_j}_r, \phi  \> dr - \int_s^t \< \Pi(u^{\epsilon, m_j}_r\cdot \nabla u^{\epsilon,m_j}_r), \phi \> dr   -\int_s^t \< \Pi(\epsilon^{H-1} w^\epsilon_r \cdot \nabla )u^{\epsilon, m_j}_r ,\phi \>dr.
 $$
 Alternatively, as shown in \cite{NSpertbyroughtransport}, $\{ u^\epsilon\}$ exist as URD solutions, and by the results in Section~\ref{sec:rough_setting} also as rough path solutions, which in the smooth case are equivalent to the usual weak solutions. URD solutions exist for Navier-Stokes perturbed by any rough transport noise, as long as the driving Rough Path is geometric, hence in particular for $X^\epsilon$ satisfying Assumption~\ref{ass:Y}.
\end{remark}
    
       \begin{lemma}\label{le:uniform-estimate}
Let $u^\epsilon$ be rough path solutions to \eqref{sf} (or resp.  \eqref{sf2}) of regularity $\frac{1}{\alpha}-var$ with initial conditions $u_0^\epsilon$ bounded uniformly in $H$.
Assume  Assumptions \ref{assumption} (resp. Assumption~\ref{ass:Y}).  Here $\alpha\in (\f 13, H)$ or as in Assumption~\ref{ass:Y}.
Then we have the uniform estimates:
        \begin{equation*}
      \sup_\epsilon \norm{u^\epsilon}_{C^{\frac{1}{\alpha}\var}([0,T];H^{-1})}^{\frac{1}{\alpha}}
        < \infty, \qquad  \sup_\epsilon \norm{R^{u^\epsilon}}_{C^{\frac{1}{2\alpha}\var}([0,T];H^{-2})}^{\frac{1}{2\alpha}} < \infty \quad a.s.
    \end{equation*}
       \end{lemma}
       \begin{proof}
   We define:
    \begin{align*}
        A^{1, \epsilon}_{s,t} = A^{(1)}(X^\epsilon_{s,t}, \cdot), \qquad
        A^{2, \epsilon}_{s,t} = A^{(2)}(\mathbb{X}^\epsilon_{s,t}, \cdot),
               \end{align*}
    which by Lemma \ref{lem:URDlinktoRP} obey:
$$
        \norm{A^{1, \epsilon}_{s,t}}^{\frac{1}{\alpha}}_{\mathcal{L}(H^{-n};H^{-(n+1)})} +
        \norm{A^{2, \epsilon}_{s,t}}^{\frac{1}{2\alpha}}_{\mathcal{L}(H^{-n};H^{-(n+2)})} \leq K^\epsilon|t-s|
$$
        where for a universal constant $C$, $$
K^\epsilon=   C(\norm{X^\epsilon}^{\frac{1}{\alpha}}_{C^{\alpha}([0,T];H^{\sigma})}+ 
        \norm{\mathbb{X}^\epsilon}^{\frac{1}{2\alpha}}_{C^{2\alpha}([0,T];H^{\sigma} \otimes H^{\sigma})}).$$

  First, as $(X^{\epsilon}, \mathbb{X}^{\epsilon})$ converges to $(B,\mathbb{B})$ in $L^{q/2}(\Omega;\mathcal{C}^{\alpha}([0,T];H^{\sigma}))$, 
    passing to a subsequence, we can assume almost sure convergence, and as a consequence of that, on a set of full probability measure
    the sample paths of $(X^{\epsilon}, \mathbb{X}^{\epsilon})$ are bounded in the rough path distance, uniformly in $\epsilon$ and in particular 
    $$K(\omega):=\sup_\epsilon K^\epsilon(\omega)<\infty.$$

      With these objects in hand, we define $\omega_{A}=K|t-s|$, and 
      $$\mu^\epsilon_{s,t} := \int_s^t \Delta u^\epsilon_r + b(u^\epsilon_r, u^\epsilon_r) dr, \quad 
        \omega_{\mu^\epsilon}(s,t) := \int_s^t (1+ \norm{u^\epsilon_r}_{H^1})^2 dr, \quad         \omega_{u^\epsilon}(s,t) := \norm{u^\epsilon}_{C^{\frac{1}{\alpha}\var}([s,t];H^{-1})}^{\frac{1}{\alpha}}
.$$

To obtain a bound on $u^\epsilon$ in $C^{\frac{1}{\alpha}\var}([0,T]; H^{-1})$ we exploit Proposition \ref{prop:globalsolnest}
    where we take ${\bf Z}={\bf X}^\epsilon$ to obtain:
    \begin{equation*}
        \omega_{u^{\epsilon}}(0,T) \lesssim \left(\frac{T}{\tilde{L}^\epsilon}\right)^{\frac{1}{\alpha}-1}(1+\norm{u^\epsilon}_{L_T^\infty H})^{\frac{1}{\alpha}}\left[ 
            \omega_{\mu^\epsilon}^{\frac{1}{\alpha}}(0,T) + \omega_{\mu^\epsilon}^{\frac{1}{3\alpha}}(0,T)\omega_{A}^{\f 13}(0,T) \right].
    \end{equation*}
    We now justify, that the right-hand side can be bounded uniformly in $\epsilon$. First, we note 
    \begin{equation*}
        \sup_\epsilon \omega_{\mu^\epsilon}(0,T) = \sup_\epsilon \int_0^T (1+ \norm{u^\epsilon_r}_{H^1})^2 dr < \infty,
    \end{equation*}
    as a consequence of the energy estimate \eqref{energyestimate}. Since $L^\epsilon \sim \f 1 {K^\epsilon}$, this has a uniform in $\epsilon$ lower bound $\frac{1}{K}$, and as $  \tilde{L}^\epsilon = \left( (K^\epsilon)^\alpha T^{2\alpha/3} + \norm{u^\epsilon}_{L^\infty_TH}^{\f 13}(K^\epsilon)^{\alpha/3}\omega_{\mu^\epsilon}^{\f 13}(0,T) \right)^{-\frac{3}{\alpha}} \wedge L^\epsilon$, the above allows us to conclude that $\tilde{L}^\epsilon$ also has a uniform in $\epsilon$ lower bound, giving $\sup_\epsilon \frac{1}{\tilde{L}^\epsilon} < \infty$.

    Combining all of the above gives:
    $$ 
    \sup_\epsilon \norm{u^\epsilon}_{C^{\frac{1}{\alpha}\var}([0,T];H^{-1})}^{\frac{1}{\alpha}}
        < \infty.
    $$

    We can obtain
    a global estimate on the remainders in a very similar way. Let us denote:
    \begin{equation*}
        R^{u^{\epsilon}}_{s,t} = u^{\epsilon}_{s,t} + \Pi \left[ X^{\epsilon}_{s,t} \cdot u^{\epsilon}_s\right], \qquad \omega_{R,\epsilon}(s,t) := \norm{R^{u^\epsilon}}_{C^{\frac{1}{2\alpha}\var}([s,t];H^{-2})}^{\frac{1}{2\alpha}}.
    \end{equation*}
    By Proposition \ref{prop:globalremainderest}:
    \begin{equation*}
        \omega_{R,\epsilon}(0,T) \lesssim  
        \left(1+ \norm{u^\epsilon}_{L_T^\infty H}^{\frac{1}{2\alpha}}\right) \left[ \omega_{\mu^\epsilon}^{\frac{1}{2\alpha}}(0,T) +\omega_A(0,T) \right]+
        \frac{T^{3/2}}{\tilde{L}^\epsilon} \omega_{u^\epsilon}(0,T).
    \end{equation*}
    Exploiting the uniform solution bound, as well as the uniform bound on $\tilde{L}^\epsilon$ obtained above,
    we conclude that the remainders are bounded 
    uniformly in epsilon in $C^{\frac{1}{2\alpha}\var}([0,T]; H^{-2})$:
    $$ 
    \sup_\epsilon \norm{R^{u^\epsilon}}_{C^{\frac{1}{2\alpha}\var}([0,T];H^{-2})}^{\frac{1}{2\alpha}} < \infty,
    $$
    completing the proof.
        \end{proof}
        
      We next state a result, which supplies a compact embedding needed for the proof. 
    \begin{lemma} \label{lem:cptembed}
        The following embedding is compact:
        \begin{equation*}
            L^\infty_T H \cap L^2_T H^1 \cap C^{\frac{1}{\alpha}\var}_T H^{-1} \hookrightarrow L^2_T H \cap C_T H_w \cap C_T H^{-1}.
        \end{equation*}
    \end{lemma}
    Proof of this can be found in Appendix A in \cite{NSpertbyroughtransport}. This particular statement is a combination of 
    simplified versions of Lemmas A2 and A3 therein. This statement can be seen as an adaptation of the classical Aubin-Lions 
    compactness Lemma to the $p$-variation spaces.
    
    Finally, we are in the position to present the proof of our main result.
 \begin{theorem} \label{general-ThmB}
 Consider the setting of Lemma~\ref{le:uniform-estimate}.
     Then the  following statements hold.
    \begin{itemize}
        \item[(i)] There exists a subsequence $\{ u^{\epsilon_n}\}$ such that $u^{\epsilon_n} \rightarrow u$ in  $L^{2}([0,T];H) \cap C([0,T],H_{w})$, 
        and the limit $u$ is a rough path solution to \eqref{eq:limit-eq}
        \begin{equation*}
            \partial_t u = \nu \Delta u- \Pi[(u \cdot \nabla) u] - \Pi[(d \mathbf{B} \cdot \nabla) u],
        \end{equation*}
        in case of (\ref{sf}) and in case of \eqref{sf2}, it solves 
         \begin{equation*}
            \partial_t u = \nu \Delta u- \Pi[(u \cdot \nabla) u] - \Pi[(d \mathbf{Z} \cdot \nabla) u].
        \end{equation*}
        
        \item[(ii)] The convergence $\{ u^{\epsilon_n}\}$ also holds in an analog of a controlled rough path distance, meaning:
        $$ u^{\epsilon_n} \rightarrow u \text{ in } C^{\frac{1}{\alpha-\delta} \var}([0,T]; H^{-1})$$
        $$ R^{u^{\epsilon_n}} \rightarrow R^u \text{ in } C^{\frac{1}{2(\alpha-\delta)}\var}([0,T]; H^{-2})$$
        hold for any $\delta>0$.
    \end{itemize}
\end{theorem}

\begin{proof}
    We will carry out the proof in three steps. We consider the case of \eqref{sf}, the other case can be proved verbatim with
     $\mathbf X^\epsilon$ replaced by $\mathbf Z^\epsilon$ and $\mathbf B$ by $\mathbf Z$.
    In the first step, we establish a collection of uniform in epsilon bounds that will be used to obtain 
    pre-compactness of the family $\{ u^\epsilon\}$ in $L^2_T H \cap C_T H_w \cap C_T H^{-1}$. In the second step we show the convergence stated in part $(ii)$.
    Finally, we complete the proof of $(i)$ by identifying the limit points as solutions to \eqref{eq:limit-eq}.

    {\bf Step 1}:
    In Lemma~\ref{le:uniform-estimate}, we obtained a uniform bound of the form:
            \begin{equation*}
    \sup_\epsilon  \norm{u^\epsilon}_{C^{\frac{1}{\alpha}\var}([0,T];H^{-1})}^{\frac{1}{\alpha}} <\infty.    \end{equation*}
Along with the classical energy estimate from Lemma~\ref{energyestimate}, we have a uniform in $\epsilon$ bound in the norm of
$L^\infty_T H \cap L^2_T H^1 \cap C^{\frac{1}{\alpha}\var}_T H^{-1} $.
Applying Lemma~\ref{lem:cptembed} we obtain a convergent subsequence in $L^2_T H \cap C_T H_w \cap C_T H^{-1}$.

    {\bf Step 2}: Now we establish convergence in $(ii)$ using Lemma~\ref{lem:pvarcptness}, 
    which states that the convergence in the rough path topology follows if we have a uniform bound in a larger $p$-variation norm and some uniform convergence.  The required bounds in finite variation spaces are contained in Lemma~\ref{le:uniform-estimate}. 
     To apply Proposition~\ref{lem:pvarcptness} we show that the subsequence from Step 1 converges uniformly to a limit $u$    in $H^{-1}$, and $R^{u^{\epsilon_n}} \rightarrow R^u$ in $H^{-2}$. The convergence of $u$ in Step 1 is already uniform. For the remainder term, we note:
    \begin{equation*}
        R^{u^{\epsilon_n}}_{s,t} = u^{\epsilon_n}_{s,t} + \Pi \left[ X^{\epsilon_n}_{s,t} \cdot u^{\epsilon_n}_s\right].
    \end{equation*}
    Now, as the path $X^{\epsilon_n}$ converges in $C^\alpha([0,T]; H^\sigma)$, it also converges uniformly in $H^\sigma$. This along with uniform convergence
    of $u^{\epsilon_n}$ in $H^{-1}$, implies the desired uniform convergence of $R^{u^{\epsilon_n}}$.
The uniform bounds of $u^\epsilon$ and $R^u$ in their respective variational topology together with the convergence concludes part $(ii)$.

By the deterministic energy bounds, as shown in Step~1, any limit point of the family $\{u^\epsilon\}$ belongs to the space $L^2_TH^1 \cap C_TH_w$. The convergence in  {Step 2} imply that the limit $u \in C^{\frac{1}{\alpha}\var}([0,T]; H^{-1})$ and $R^u \in C^{\frac{1}{2\alpha}\var}([0,T]; H^{-2})$.

    {\bf Step 3}: 
 To show that the limit is a rough path solution, it remains to show that it satisfies the weak formulation \eqref{eq:weak_formulation}.
 Firstly,  for any $\phi\in H^\sigma$,  $u^{\epsilon_n}$ satisfies:
    \begin{equation}
        \< u^{\epsilon_n}_{s,t}, \phi \> = \int_s^t \<\nu \Delta u^{\epsilon_n}_r, \phi  \> dr - \int_s^t \< \Pi(u^{\epsilon_n}_r\cdot \nabla u^{\epsilon_n}_r), \phi \> dr 
        -\int_s^t \< \Pi(d{\bf X}^{\epsilon^n}_r \cdot \nabla )u^{\epsilon_n}_r ,\phi \>.
    \end{equation}
    Using the convergence established above, we see:
    \begin{align*}
     & |  \< u^{\epsilon_n}_{s,t}, \phi \> - \< u_{s,t}, \phi \>|\le \norm{u^{\epsilon_n}_{s,t} - u_{s,t}}_{H^{-3}} \norm{\phi}_{H^3}\,\\
      & \Bigl|  \int_s^t \<\nu \Delta u^{\epsilon_n}_r, \phi  \> dr - \int_s^t \<\nu \Delta u_r, \phi  \> dr\Bigr|\lesssim \int_s^t \norm{u^{\epsilon_n}_r-u_r}_{H^{-1}}\norm{\phi}_{H^3},\\
      &{} \Bigl| \int_s^t \< u^{\epsilon_n}_r\cdot \nabla u_r^{\epsilon_n} -u\cdot \nabla u , \phi \> dr \Bigr| 
       \le \int_s^t \norm{u^{\epsilon_n}_r-u_r}_{H} \; \norm{u_r^{\epsilon_n} \cdot \nabla \phi }_{H}dr+\int_s^t \norm{u_r}_H\norm{ (u_r^{\epsilon_n}-u_r) \cdot \nabla \phi}_Hdr\\
       & \lesssim  \norm{u^{\epsilon_n}-u}_{L^2([s,t];H)}  \norm{\phi}_{H^3} \norm{u^{\epsilon_n}}_{L^2([s,t];H)} + \norm{u}_{L^2([s,t];H)} \norm{\phi}_{H^3} \norm{u^{\epsilon_n}-u}_{L^2([s,t];H)}.  
    \end{align*}
 The convergence follows from the embedding indicated in Lemma~\ref{lem:cptembed}.
The convergence of the final, rough integral term follows from the continuity of rough integration
    (see Theorem 4.17 in \cite{FH}) in the controlled setting. 
    Recalling the notation from Proposition \ref{rem:distrvstested}, replacing $u$ by $u^{\epsilon_n}$ and $\mathbf Z$ by 
    ${\bf X}^{\epsilon_n}$, we have:
    \begin{align*}
        \left| \int_s^t \< \Pi(d{\bf X}^{\epsilon_n}_r \cdot \nabla )u^{\epsilon_n}_r ,\phi \> - \int_s^t \< \Pi(d{\bf B}_r \cdot \nabla )u_r ,\phi \>  \right| &\lesssim \norm{\phi}_{H^3} \Bigl( \norm{u^{\epsilon_n} -u }_{C^{\frac{1}{\alpha} \var}([0,T];H^{-1})} \\ &+ \norm{R^{u^{\epsilon_n}} -R^u}_{C_2^{\frac{1}{2\alpha}\var}([0,T];H^{-2})} + 
        \rho_\alpha({\bf X}^{\epsilon_n}, {\bf B})\Bigr).
    \end{align*}
  By part Step 2 of the proof, and Theorem \ref{thm:A}, the above convergence is justified. This allows us to pass to the limit in the weak formulation and finishes the proof.
    \end{proof}

\begin{remark}
    By theorem \ref{thm:equivalence-of-sol}, the Rough Path, and URD solutions are equivalent. We nevertheless provide the proof of
    the last step in the URD setting as well.
    Let us recall \eqref{eq:URDweak_formulation}:
    \begin{equation*}
        \< u^\epsilon_{s,t}, \phi \> = \int_s^t \<\nu \Delta u^\epsilon_r, \phi  \> dr + \int_s^t \<u^\epsilon_r \cdot \nabla u^\epsilon_r, \phi \> dr + 
        \< A^{1,\epsilon}_{s,t}u_s,\phi \> + \< A^{2,\epsilon}_{s,t}u_s, \phi\> + \< u^{\natural,\epsilon}_{s,t}, \phi \>.
    \end{equation*}
    First we will justify the pointwise in time convergence in $H^{-3}$ of all terms,
    and then we will assert that the limit remainder is of the required regularity. The convergence of the first three terms are as in the proof above.
    Recalling the notation from Proposition \ref{rem:distrvstested}, replacing $u$ by $u^{\epsilon_n}$ and $\mathbf Z$ by 
    ${\bf X}^{\epsilon_n}$, we have:
\begin{equation*}
    \< A^{1,\epsilon_n}_{s,t}u_s,\phi \> = y^{\epsilon_n}_s X^{\epsilon_n}_{s,t}, \qquad \< A^{2,\epsilon_n}_{s,t}u_s, \phi\> = y^{\prime, \epsilon_n}_s \mathbb{X}^{\epsilon_n}_{s,t}.
\end{equation*}
Then the convergence of these terms is once again deduced from Theorem~\ref{thm:A}, and convergence of $u^{\epsilon_n}$.
     This in particular means that the residual converges as well.
    Let us define $u^\natural_{s,t} = \lim_{\epsilon \rightarrow 0}u_{s,t}^{\natural,\epsilon}$ in $H^{-3}$.
       We also denote:
    \begin{equation*}
        \omega_{\natural,\epsilon}(s,t):= \norm{u^{\natural,\epsilon}}^{\frac{1}{3\alpha}}_{C^{\frac{1}{3\alpha} - var}([s,t];H^{-3})}.
    \end{equation*}

    We now  justify the regularity in time in the sense of local finite $\frac{1}{3\alpha}$-variation.
By \eqref{est:residual}, for $|t-s|\leq L$ for $L$ some random constant:
    \begin{align*}
        \norm{u^{\epsilon, \natural}}^{\frac{1}{3\alpha}}_{C^{\frac{1}{3\alpha} - var}([s,t];H^{-3})} = \omega_{\epsilon, \natural}(s,t) &\lesssim 
        \norm{u^{\epsilon}}_{L^{\infty}_{T}H}^{\frac{1}{3\alpha}} \omega_{A}(s,t) +
        \omega_{\mu^{\epsilon}}^{\frac{1}{3\alpha}}(s,t)\omega_{A}^{\f 13}(s,t) \\
        &\lesssim \norm{u^{\epsilon}_{0}}^{\f 23\alpha}_{H} \leq M
    \end{align*}
where M is some random constant independent of $\epsilon$. Consequently:
\begin{align*}
    |u^{\natural}_{s,t}(\phi)|^{\frac{1}{3\alpha}}
    = \lim_{\epsilon \rightarrow 0}|u^{\natural, \epsilon}_{s,t}(\phi)|^{\frac{1}{3\alpha}} \leq \norm{\phi}_{H^{3}}^{\frac{1}{3\alpha}}\;
    \liminf_{\epsilon \rightarrow 0}
    \norm{u^{\epsilon, \natural}}^{\frac{1}{3\alpha}}_{C^{\frac{1}{3\alpha} - var}([s,t];H^{-3})} 
    \leq \norm{\phi}_{H^{3}}^{\frac{1}{3\alpha}} \;\liminf_{\epsilon \rightarrow 0}
    \omega_{\epsilon, \natural}(s,t).
\end{align*}
Hence, any $(s,t)$ with $|t-s|\leq L$ and a partition $\pi$ of this interval $[s,t]$ we have:
\begin{align*}
    \sum_{[u,v]\in \pi}\norm{u^{ \natural}_{u,v}}^{\frac{1}{3\alpha}}_{H^{-3}} \leq \liminf_{\epsilon \rightarrow 0}  \sum_{[u,v]\in \pi}
    \omega_{\epsilon, \natural}(u,v) \leq \liminf_{\epsilon \rightarrow 0} \omega_{\epsilon, \natural}(s,t) \leq M.
\end{align*}
Taking supremum over all partitions we obtain $\norm{u^{\natural}}^{\frac{1}{3\alpha}}_{C^{\frac{1}{3\alpha} - var}([s,t];H^{-3})}\lesssim M$.
This shows that $u^{\natural}$ is of finite $\frac{1}{3\alpha}$-variation on intervals of length $L$. Hence, 
$u$ is a URD solution to \eqref{eq:limit-eq}, so also a rough path solution.
\end{remark}
\section{Discussion on non-Gaussian inputs}

In principle it is possible to extend the main results to the non-Gaussian Hermite process driven fast motion. They are self similar, have stationary increments and live on a fixed Wiener chaos. 
Hermite processes of higher order are defined using kernel representations. For $H>\f 12$,  see e.g. \cite{tudor-hermite},
this is:
$$Z_t^{H,m}=\int_{\R^m} \int_0^t \Pi_{j=1}^m (s-\xi_j)_+^{H_0-\f 32} ds \;dW_{\xi_1}\dots dW_{\xi_m},$$
in which case the theory of rough paths is not needed, and our arguments work exactly as for the case of fBM with $H\in(\frac{1}{2}, 1)$.

 The more interesting case is for $H<\frac{1}{2}$.  A class of generalised Hermite processes is defined  in \cite{Assaad-HOU}:
 \begin{align*}
    X^{(q,H,\beta)}_{t} &= C(q,H,\beta)\int_{\mathbb{R}^q} \left[ \int_{\mathbb{R}} g_t^\beta(u)\prod_{i=1}^q 
    (u-y_i)_{+}^{-(\frac{1}{2}+\frac{1-H}{q})} du \right] \times dB(y_1) \dots dB(y_q),
\end{align*}
where $g_t^\beta(u)= (t-u)_+^\beta - (-u)_+^\beta$ for $\beta \neq 0$.
The parameters are chosen so that $H+\beta \in (0,1)$ and $H\in (\frac{1}{2}, 1)$. 

The process $ X^{(q,H,\beta)}$ has stationary increments,
and is self-similar of order $H+\beta$, so its covariance function coincides with that of fBM with Hurst parameter $H+\beta$. An application 
of Kolmogorov criterion guarantees existence of a version with almost sure H\"older continuous paths of order $(H+\beta)-$.

One can then ask for a solution to the Langevin equation driven by this process:
$$ dw_t = -c w_t dt + \lambda d X^{(q,H,\beta)}_t, $$ 
which admits a unique solution with an explicit expression $$ w_t = w_0e^{-ct} + \sigma \int_0^t e^{-c(t-s)} d X^{(q,H,\beta)}_s, $$ with the integral understood in
the Riemann-Stieltjes sense (cf Chapter 3.3 in \cite{tudor-hermite}). Moreover, there exists a stationary solution, given by:
$$ w_t = \sigma \int_{-\infty}^{t}e^{-c(t-s)}d X^{(q,H,\beta)}_s. $$
Crucial for our application is that the covariance of this process enjoys the same asymptotics as the fOU process:
$\E[w_{t+h}w_t]\sim h^{2(H+\beta) -2}$ as $h \to \infty$, see \cite[Prop. 6]{Assaad-HOU}.

Let $w^\epsilon$ be the rescaled stationary HOU process, defined as:
\begin{equation}
    dw^\epsilon_t = -\epsilon^{-1} M w^\epsilon_t dt + 
   \epsilon^{-H} Q^{\f 12} dX^{(q,H,\beta)}_t,
\end{equation}
with the equation understood componentwise and where $ Q^{\frac{1}{2}} X^{(q,H,\beta)}_t= \sum_{i \geq 1}\sqrt{\lambda_i}e_i Q_t^{H,i}$
and $\{Q_t^{H,i}\}$ are independent copies.

\medskip

{\it Conjecture.} 
\label{lemma-Hermite} 
Let $M$ and $Q$ satisfy Assumption \ref{assumption}, $H\in \left(\frac{1}{3}, \frac{1}{2}\right)$, and
   \begin{equation*}
       X_{s,t}^\epsilon=\epsilon^{H-1} \int_s^t  w_r^\epsilon dr, \quad \mathbb{X}_{s,t}^\epsilon = 
       \int_s^t X_{s,r}^\epsilon \otimes dX_{r}^\epsilon
   \end{equation*}
Then  $(X^{\epsilon}_{\cdot}, \mathbb{X}^{\epsilon}_{ \cdot\cdot})$ converges in $L^{q}(\Omega;\mathcal{C}^{\alpha}
   ([0,T];H^{\sigma}))$ to $(Z,\mathbb{Z})$ for all $q>1$,  and $\frac{1}{3}<\alpha< H$.

\medskip

Suppose that the conjecture holds, $(X^{\epsilon}_{\cdot}, \mathbb{X}^{\epsilon}_{ \cdot\cdot})$ converges, there is only one possibility for the limit.
Firstly, on the first level we have $Z= M^{-1}Q^{\frac{1}{2}}X^{(q,H,\beta)}$. In fact,  for the Young integral case, a functional CLT is known \cite{GehPhD} for 1d processes. 

Our focus is on generalised Hermite processes with regularity in $(\f 13, \f 12)$ and their rough path lifts.
The lift can be constructed 
using the same procedure as for a Gaussian lift. This is not obvious, as the generalised Hermite process is not Gaussian, but
we can still construct a lift. We note that the construction of iterated integrals as presented in Section \ref{subsec:pf-thmA}
applies provided we can apply the necessary properties of the covariance.

To guarantee the finite $2$-d variation of the covariance for $X^{(q,H,\beta)}$, one can use Lemma 10.8 and Theorem 10.9 in \cite{FH},
which never use the assumed Gaussianity. Then we essentially apply Gaussian lift theorem, using hypercontractivity instead of 
Gaussianity to argue equivalence of moments, to deduce the H\"older regularity.

Inspecting the proof of {Theorem~\ref{thm:CQ}} 
provided in \cite{FVroughpaths}, we see 
that Gaussianity is actually never used in this result. For this reason, this Theorem still applies to the components of the
processes $X_{s,t}^\epsilon$ $M^{-1}Q^{\frac{1}{2}}X^{(q,H,\beta)}$, and $w$. Thanks to identical covariance asymptotics of HOU and fOU, verifying 
assumptions of \ref{thm:CQ} for $X_{s,t}^\epsilon$ follows exactly the same as before.

We are left to deal with the covariance of $w$.  We expect:

{\it Conjecture 2.} $$  \norm{R_{X^{\epsilon, i}- M^{-1}Q^{\f 12}X^{(q,H,\beta,)}_{t}}}_{\frac{1}{2H}\var} \lesssim \lambda_i |t-s|^{2H}. $$

Conjecture 2 implies Conjecture 1.

If Conjecture 1 holds, i.e. Theorem \ref{thm:A} extends to the Hermite case, the corresponding Theorem~\ref{thm:B} follows at once, from \ref{general-ThmB}. 
The limit equation is again of the form \eqref{eq:modeleqn}  driven by a Hermite OU fast process. 
For the proof, we use Theorem~\ref{lemma-Hermite} instead of Theorem~\ref{thm:A}.


\begin{thebibliography}{MTVE01}

\bibitem[ABC{\etalchar{+}}23]{ABCGG}
Gabriel~B. Apolinário, Geoffrey Beck, Laurent Chevillard, Isabelle Gallagher,
  and Ricardo Grande.
\newblock A linear stochastic model of turbulent cascades and fractional
  fields.
\newblock 2301.00780, 2023.

\bibitem[ADT23]{Assaad-HOU}
Obayda Assaad, Charles-Phillipe Diez, and Ciprian~A. Tudor.
\newblock Generalized wiener–hermite integrals and rough non-gaussian
  ornstein–uhlenbeck process.
\newblock {\em Stochastics}, 95(2):191--210, 2023.

\bibitem[Bat99]{Batchelor}
G.~K. Batchelor.
\newblock {\em An introduction to fluid dynamics}.
\newblock Cambridge Mathematical Library. Cambridge University Press,
  Cambridge, paperback edition, 1999.

\bibitem[BBC{\etalchar{+}}24]{numerical-frac}
Geoffrey Beck, Charles-Edouard Br\'ehier, Laurent Chevillard, Ricardo Grande,
  and Wandrille Ruffenach.
\newblock Numerical simulations of a stochastic dynamics leading to cascades
  and loss of regularity: Applications to fluid turbulence and generation of
  fractional gaussian fields.
\newblock {\em Phys. Rev. Res.}, 6:033048, Jul 2024.

\bibitem[BF13]{Boyer-Fabrie-13-book}
Franck Boyer and Pierre Fabrie.
\newblock {\em Mathematical tools for the study of the incompressible
  {N}avier-{S}tokes equations and related models}, volume 183 of {\em Applied
  Mathematical Sciences}.
\newblock Springer, New York, 2013.

\bibitem[BG17]{URDog}
Ismael Bailleul and Massimiliano Gubinelli.
\newblock Unbounded rough drivers.
\newblock {\em Annales de la Faculté des sciences de Toulouse}, 26, 2017.

\bibitem[BGS21]{Bourguin-Gailus-Spiliopoulos-typical}
Solesne Bourguin, Siragan Gailus, and Konstantinos Spiliopoulos.
\newblock Typical dynamics and fluctuation analysis of slow–fast systems
  driven by fractional brownian motion.
\newblock {\em Stochastics and Dynamics}, 21(07):2150030, 2021.

\bibitem[BT05]{Boufoussi-Tudor05}
Brahim Boufoussi and Ciprian~A. Tudor.
\newblock Kramers-{S}moluchowski approximation for stochastic evolution
  equations with {FBM}.
\newblock {\em Rev. Roumaine Math. Pures Appl.}, 50(2):125--136, 2005.

\bibitem[BT16]{not-mixing}
Shuyang Bai and Murad~S. Taqqu.
\newblock Short-range dependent processes subordinated to the {G}aussian may
  not be strong mixing.
\newblock {\em Statist. Probab. Lett.}, 110:198--200, 2016.

\bibitem[CF09]{Caruana-Friz-09}
Michael Caruana and Peter Friz.
\newblock Partial differential equations driven by rough paths.
\newblock {\em J. Differential Equations}, 247(1):140--173, 2009.

\bibitem[CPM03]{fou_cheridito}
Kawaguchi~Hideyuki Cheridito~Patrick and Maejima Makoto.
\newblock Fractional ornstein-uhlenbeck processes.
\newblock {\em Electronic Journal of Probability}, 8, 2003.

\bibitem[CQ02]{coutin-qian}
Laure Coutin and Zhongmin Qian.
\newblock Stochastic analysis, rough path analysis and fractional brownian
  motions.
\newblock {\em Probability Theory and Related Fields}, 122, 2002.

\bibitem[Dav08]{davie}
A.~M. Davie.
\newblock Differential equations driven by rough paths: An approach via
  discrete approximation.
\newblock {\em Applied Mathematics Research eXpress}, 2008, 2008.

\bibitem[DH23]{HofDeb23}
Arnaud Debussche and Martina Hofmanov\'a.
\newblock Rough analysis of two scale systems.
\newblock {\em arXiv preprint arXiv:2306.15781}, 2023.

\bibitem[DP24]{DebPapMultiscale}
Arnaud Debussche and Umberto Pappalettera.
\newblock Second order perturbation theory of two-scale systems in fluid
  dynamics.
\newblock {\em Eur. Math. Soc.}, 2024.

\bibitem[EKN20]{Eichinger_2020}
Katharina Eichinger, Christian Kuehn, and Alexandra Neamţu.
\newblock Sample paths estimates for stochastic fast-slow systems driven by
  fractional brownian motion.
\newblock {\em Journal of Statistical Physics}, 179(5–6):1222–1266, January
  2020.

\bibitem[EM02]{Embrechts-Maejima}
Paul Embrechts and Makoto Maejima.
\newblock {\em Selfsimilar processes}.
\newblock Princeton Series in Applied Mathematics. Princeton University Press,
  Princeton, NJ, 2002.

\bibitem[FF23]{franco-frac}
Russo~F. Flandoli~F.
\newblock Reduced dissipation effect in stochastic transport by gaussian noise
  with regularity greater than 1/2.
\newblock {\em arXiv preprint arXiv:2305.19293}, 2023.

\bibitem[FGL15]{Friz-Gassiat-Lyons15}
Peter Friz, Paul Gassiat, and Terry Lyons.
\newblock Physical {B}rownian motion in a magnetic field as a rough path.
\newblock {\em Trans. Amer. Math. Soc.}, 367(11):7939--7955, 2015.

\bibitem[FK00]{Fannjiang-Komorowski-2000}
Albert Fannjiang and Tomasz Komorowski.
\newblock Fractional {B}rownian motions in a limit of turbulent transport.
\newblock {\em Ann. Appl. Probab.}, 10(4):1100--1120, 2000.

\bibitem[FNS20]{Friz-Nilsssen-Stannat}
Peter~K. Friz, Torstein Nilssen, and Wilhelm Stannat.
\newblock Existence, uniqueness and stability of semi-linear rough partial
  differential equations.
\newblock {\em J. Differential Equations}, 268(4):1686--1721, 2020.

\bibitem[FP21]{Flandoli-pap2021}
Franco Flandoli and Umberto Pappalettera.
\newblock 2d euler equations with stratonovich transport noise as a large-scale
  stochastic model reduction.
\newblock {\em Journal of Nonlinear Science}, 31(1):1432--1467, 2021.

\bibitem[FP22]{Flandoli-pap2022}
Franco Flandoli and Umberto Pappalettera.
\newblock From additive to transport noise in 2d fluid dynamics.
\newblock {\em Stochastics and Partial Differential Equations: Analysis and
  Computations}, 10(3):964--1004, 2022.

\bibitem[FV10]{FVroughpaths}
Peter~K. Friz and Nicolas~B. Victoir.
\newblock {\em Multidimensional stochastic processes as rough paths}.
\newblock Cambridge University Press, 2010.

\bibitem[GcH19]{Gerasimovics-Hairer}
Andris Gerasimovi\v~cs and Martin Hairer.
\newblock H\"ormander's theorem for semilinear {SPDE}s.
\newblock {\em Electron. J. Probab.}, 24:Paper No. 132, 56, 2019.

\bibitem[Geh22]{GehPhD}
Johann Gehringer.
\newblock {\em Stochastic Homogenization of Fast-Slow Systems Driven by
  Fractional Noise}.
\newblock PhD thesis, Imperial College London, 2022.

\bibitem[GL19]{gehringer-Li-19}
Johann Gehringer and Xue-Mei Li.
\newblock Homogenization with fractional random fields, 2019.

\bibitem[GL20]{gehringer-li2}
Johann Gehringer and Xue-Mei Li.
\newblock Diffusive and rough homogenisation in fractional noise field, 2020.

\bibitem[GL22]{Gehringer-Li-fOU}
Johann Gehringer and Xue-Mei Li.
\newblock Functional limit theorems for the fractional {O}rnstein-{U}hlenbeck
  process.
\newblock {\em J. Theoret. Probab.}, 35(1):426--456, 2022.

\bibitem[GLS21]{volterra}
Johann Gehringer, Xue-Mei Li, and Julian Sieber.
\newblock Functional limit theorems for volterra processes and applications to
  homogenization.
\newblock {\em Nonlinearity}, 35, 2021.

\bibitem[GT10]{Gubinelli-Tindel-10}
Massimiliano Gubinelli and Samy Tindel.
\newblock Rough evolution equations.
\newblock {\em Ann. Probab.}, 38(1):1--75, 2010.

\bibitem[HF20]{FH}
Martin Hairer and Peter~K. Friz.
\newblock {\em A course on Rough Paths}.
\newblock Springer, 2020.

\bibitem[HH18]{EnergymethodRPDEs}
Martina Hofmanov\'a and Antoine Hocquet.
\newblock An energy method for rough partial differential equations.
\newblock {\em JOURNAL OF DIFFERENTIAL EQUATIONS}, 265(4), 2018.

\bibitem[HL22]{Hairer:22}
Martin Hairer and Xue-Mei Li.
\newblock Generating diffusions with fractional brownian motion.
\newblock {\em Communications in Mathematical Physics}, 396(1):91–141, August
  2022.

\bibitem[HLN19]{NSpertbyroughtransport}
Martina Hofmanov\'a, James-Michael Leahy, and Torstein Nilssen.
\newblock On the navier-stokes equation perturbed by rough transport noise.
\newblock {\em Journal of Evolution Equations}, 19(1), 2019.

\bibitem[HN21]{ItoRPDE}
Antoine Hocquet and Torstein Nilssen.
\newblock An itoˆ formula for rough partial differential equations and some
  applications.
\newblock {\em Potential Analysis}, 54, 2021.

\bibitem[KNR14]{Komorowski-Novikov-Ryzhik14}
Tomasz Komorowski, Alexei Novikov, and Lenya Ryzhik.
\newblock Homogenization driven by a fractional brownian motion: The shear
  layer case.
\newblock {\em Multiscale Modeling \& Simulation}, 12(2):440--457, 2014.

\bibitem[LS22]{Li-Sieber-mild}
Xue-Mei Li and Julian Sieber.
\newblock Mild stochastic sewing lemma, {SPDE} in random environment, and
  fractional averaging.
\newblock {\em Stoch. Dyn.}, 22(7):Paper No. 2240025, 47, 2022.

\bibitem[MP08]{infdimfOU}
Bohdan Maslowski and Jan Pospisil.
\newblock Ergodicity and parameter estimates for infinite-dimensional
  fractional ornstein-uhlenbeck process.
\newblock {\em Applied Mathematics and Optimization}, 57(3), 2008.

\bibitem[MTVE01]{Majda:01}
Andrew~J. Majda, Ilya Timofeyev, and Eric Vanden~Eijnden.
\newblock A mathematical framework for stochastic climate models.
\newblock {\em Communications on Pure and Applied Mathematics}, 54(8):891--974,
  2001.

\bibitem[NNT10]{Nourdin-Nualart-Tudor}
Ivan Nourdin, David Nualart, and Ciprian~A. Tudor.
\newblock {Central and non-central limit theorems for weighted power variations
  of fractional Brownian motion}.
\newblock {\em Annales de l'Institut Henri Poincaré, Probabilités et
  Statistiques}, 46(4):1055 -- 1079, 2010.

\bibitem[NNZ16]{Nourdin-Nualart-Zintout-Rola}
Ivan Nourdin, David Nualart, and Rola Zintout.
\newblock Multivariate central limit theorems for averages of fractional
  {V}olterra processes and applications to parameter estimation.
\newblock {\em Stat. Inference Stoch. Process.}, 19(2):219--234, 2016.

\bibitem[PT00]{Pipiras-Taqqu00}
Vladas Pipiras and Murad~S. Taqqu.
\newblock Integration questions related to fractional {B}rownian motion.
\newblock {\em Probab. Theory Related Fields}, 118(2):251--291, 2000.

\bibitem[PT17]{Pipiras-Taqqu-book}
Vladas Pipiras and Murad~S. Taqqu.
\newblock {\em Long-range dependence and self-similarity}.
\newblock Cambridge Series in Statistical and Probabilistic Mathematics, [45].
  Cambridge University Press, Cambridge, 2017.

\bibitem[QX24]{Qian-Xu-2024}
Zhongmin Qian and Xingcheng Xu.
\newblock Lévy area analysis and parameter estimation for fou processes via
  non-geometric rough path theory.
\newblock {\em Acta Mathematica Scientia}, 44(5), 2024.

\bibitem[Sam06]{Samorodnitsky}
Gennady Samorodnitsky.
\newblock Long range dependence.
\newblock {\em Found. Trends Stoch. Syst.}, 1(3):163--257, 2006.

\bibitem[Soh01]{Sohr-book}
Hermann Sohr.
\newblock {\em The {N}avier-{S}tokes equations}.
\newblock Modern Birkh\"auser Classics. Birkh\"auser/Springer Basel AG, Basel,
  2001.
\newblock An elementary functional analytic approach, [2013 reprint of the 2001
  original] [MR1928881].

\bibitem[Tud23]{tudor-hermite}
Ciprian Tudor.
\newblock {\em Non-Gaussian Selfsimilar Stochastic Processes}.
\newblock Springer Cham, 2023.

\bibitem[XXYP25]{xu2025largedeviationprincipleslowfast}
Wenting Xu, Yong Xu, Xiaoyu Yang, and Bin Pei.
\newblock Large deviation principle for slow-fast systems with
  infinite-dimensional mixed fractional brownian motion, 2025.

\end{thebibliography}
\newcommand{\etalchar}[1]{$^{#1}$}

\end{document}